\documentclass [12pt]{article}
\usepackage[]{amsmath,amssymb,amsfonts,latexsym,amsthm,enumerate,paralist}
\usepackage[numeric,initials,nobysame]{amsrefs}

\long\def\symbolfootnote[#1]#2{\begingroup%
\def\thefootnote{\fnsymbol{footnote}}\footnote[#1]{#2}\endgroup}

\def\ex{\mathrm{ex}}
\newcommand{\N}{\mathbb{N}}

\def\E{\mathcal{E}}
\def\G{\mathcal{G}}

\def\GG{\mathbb{G}}
\def\MM{\mathbb{M}}

\def\expec#1{{\mathbb E[#1]}\,}

\def\eventcnd#1#2{\{#1\,|\,#2\}}

\def\event#1{\{#1\}}

\def\probcnd#1#2{\textrm{Pr}[#1\,|\,#2]}
\def\prob#1{\textrm{Pr}[#1]}
\def\Prob#1{\textrm{Pr}\Big[#1\Big]}
\def\Probcnd#1#2{\textrm{Pr}\Big[#1\,\Big|\,#2\Big]}
\def\ol#1{\overline{#1}}
\def\floor#1{\lfloor #1 \rfloor}

\def\ceil#1{\lceil #1 \rceil}
\def\ex{\textrm{ex}}

\date{
}

\title{Lower Bounds for the Size of Random Maximal $H$-Free Graphs}
\author{Guy Wolfovitz \\
\small Department of Computer Science\\[-0.8ex]
\small Haifa University, Haifa, Israel\\[-0.8ex]
\small \texttt{gwolfovi@cs.haifa.ac.il}}
\newtheorem{theorem}{Theorem}[section]
\newtheorem{lemma}[theorem]{Lemma}
\newtheorem{claim}[theorem]{Claim}
\newtheorem{definition}{Definition}

\newtheorem{proposition}[theorem]{Proposition}

\newtheorem{fact}[theorem]{Fact}
\renewcommand{\epsilon}{\varepsilon}

\newtheoremstyle{upright}%
        {8pt plus2pt minus4pt}%
        {8pt plus2pt minus4pt}%
        {\upshape}%
        {}%
        {\bfseries}%
        {:}%
        {1em}%
        {}%
\theoremstyle{upright}

\newcommand{\ignore}[1]{}
\begin{document}

\maketitle

\begin{abstract}
We consider the next random process for generating a maximal $H$-free
graph: Given a fixed graph $H$ and an integer $n$, start by taking a
uniformly random permutation of the edges of the complete $n$-vertex
graph $K_n$. Then, traverse the edges of $K_n$ according to the order
imposed by the permutation and add each traversed edge to an (initially
empty) evolving $n$-vertex graph - unless its addition creates a copy of
$H$. The result of this process is a maximal $H$-free graph $\MM_n(H)$.
Our main result is a new lower bound on the expected number of edges in
$\MM_n(H)$, for $H$ that is regular, strictly $2$-balanced.
As a corollary, we obtain new lower bounds for Tur\'{a}n numbers of
complete, balanced bipartite graphs.  Namely, for fixed $r \ge 5$, we
show that $\ex(n, K_{r,r}) = \Omega(n^{2-2/(r+1)}(\ln\ln n)^{1/(r^2-1)})$.
This improves an old lower bound of Erd\H{o}s and Spencer.

Our result relies on giving a non-trivial lower bound on the probability
that a given edge is included in $\MM_n(H)$, conditioned on the event
that the edge is traversed relatively (but not trivially) early during
the process.
\end{abstract}

\vspace{.21cm}

\section{Introduction}
\label{sec::1}
Consider the next random process for generating a maximal $H$-free
graph. Given $n \in \N$ and a graph $H$, assign every edge $f$ of the
complete $n$-vertex graph $K_n$ a birthtime $\beta(f)$, distributed
uniformly at random in the interval $[0,1]$. (Note that with probability $1$ the
birthtimes are distinct and so $\beta$ is a permutation.) Now start with
the empty $n$-vertex graph and iteratively add edges to it as follows.
Traverse the edges of $K_n$ in order of their birthtimes, starting with
the edge whose birthtime is smallest, and add each traversed edge to
the evolving graph, unless its addition creates a copy of $H$.  When all
edges of $K_n$ have been exhausted, the process ends. Denote by $\MM_n(H)$
the graph which is the result of the above process.
The main concern in this paper is bounding from below the expected number
of edges of $\MM_n(H)$, which is denoted by $e(\MM_n(H))$. We always
think of $H$ as being fixed and of $n$ as going to $\infty$. To be able to
state our results, we need a few definitions. For a graph $H$, let $v_H$
and $e_H$ denote, respectively, the number of vertices and edges in $H$.
Say that a graph $H$ is \emph{strictly $2$-balanced} if $v_H, e_H \ge
3$ and for every $F \subsetneq H$ with $v_F \ge 3$, $(e_H-1)/(v_H-2) >
(e_F-1)/(v_F-2)$. Examples of strictly $2$-balanced graphs include the
$r$-cycle $C_r$, the complete $r$-vertex graph $K_r$, the
complete bipartite graph $K_{r-1,r-1}$
and the $(r-1)$-dimensional cube, for all $r \ge 3$. Note that all of these
examples are of graphs which are regular.
Our main result follows.
\begin{theorem}
\label{thm::main}
Let $H$ be a regular, strictly $2$-balanced graph. Then 
\begin{displaymath}
\expec{e(\MM_n(H))}  =  \Omega\big( n^{2 - (v_H-2)/(e_H-1)}
           (\ln\ln n)^{1/(e_H-1)} \big).
\end{displaymath}
\end{theorem}

Before discussing what was previously known about $e(\MM_n(H))$, we
state an immediate consequence of Theorem~\ref{thm::main} in extremal
graph theory.  Let $\ex(n,H)$ be the largest integer $m$ such that there
exists an $H$-free graph over $n$ vertices and $m$ edges.  For the
case where $H = K_{r,r}$, K\H{o}v\'{a}ri, S\'{o}s and Tur\'{a}n
proved that for fixed $r$, $\ex(n,K_{r,r}) = O(n^{2-1/r})$. For
$r\in\{2,3\}$ this upper bound is known to be tight, by explicit
constructions, due to Erd\H{o}s, R\'{e}nyi and S\'{o}s~\cite{ERS}
and Brown~\cite{Brown}. Since $\ex(n,K_{4,4}) \ge \ex(n,K_{3,3})$,
one has that $\ex(n,K_{4,4})  =  \Omega(n^{2-1/3})$. For fixed $r \ge
5$, Erd\H{o}s and Spencer~\cite{erdos74probabilistic} used a simple
application of the probabilistic method to prove $\ex(n,K_{r,r})  =
\Omega(n^{2-2/(r+1)})$.  Now note that Theorem~\ref{thm::main} implies a
lower bound for $\ex(n,H)$ for every regular, strictly $2$-balanced graph.
Hence, since $K_{r,r}$ is regular and strictly $2$-balanced, we obtain
the next lower bound on $\ex(n,K_{r,r})$ which improves asymptotically
the lower bound of Erd\H{o}s and Spencer for $r \ge 5$.
\begin{theorem}
For all $r \ge 5$, $\ex(n,K_{r,r}) =
\Omega\big(n^{2-2/(r+1)}(\ln\ln n)^{1/(r^2-1)}\big)$.
\end{theorem}

\subsection{Previous bounds on $e(\MM_n(H))$}
\label{sec:previous}
The first to investigate the number of edges in $\MM_n(H)$ were
Ruci\'{n}ski and Wormald~\cite{RucinskiW92}, who considered the case
where $H=K_{1,r+1}$ is a star with $r+1$ edges. In that case, it was shown
than with probability approaching $1$ as $n$ goes to infinity, $\MM_n(H)$
is an extremal $H$-free graph (that is, every vertex in $\MM_n(H)$ has
degree exactly $r$, except perhaps for at most one vertex whose degree
is $r-1$).
Erd\H{o}s, Suen and Winkler~\cite{ErdosSW95} showed that
with probability that goes to $1$ as $n$ goes to $\infty$, 
$e(\MM_n(K_3)) = \Omega(n^{3/2})$.
Bollob\'{a}s and Riordan~\cite{Boll00} considered the case of $H \in
\{K_4,C_4\}$, and showed that with probability that goes to $1$ as $n$
goes to $\infty$, $e(\MM_n(K_4)) = \Omega(n^{8/5})$ and $e(\MM_n(C_4)) =
\Omega(n^{4/3})$. Osthus and Taraz~\cite{OsthusT01} generalized these
bounds for every strictly $2$-balanced graph $H$, showing that with
probability that goes to $1$ as $n$ goes to $\infty$, $e(\MM_n(H))
= \Omega(n^{2-(v_H-2)/(e_H-1)})$.  Note that the above lower bounds
trivially imply similar
lower bounds on the expectation of $e(\MM_n(H))$.
It is worth mentioning that all of the above lower bounds on the
expectation of $e(\MM_n(H))$ can be derived using standard correlation
inequalities.

The first
non-trivial lower bound on the expectation of $e(\MM_n(H))$ for some
graph $H$ that contains a cycle
was given by Spencer~\cite{Spencer0a}. Spencer showed that
for every constant $a$ there exists $n_0 = n_0(a)$ such that for every
$n \ge n_0$, $\expec{e(\MM_n(K_3))} \ge an^{3/2}$. In the same paper,
Spencer conjectured that $\expec{e(\MM_n(K_3))} = \Theta( n^{3/2}(\ln
n)^{1/2})$.
Recently, Bohman~\cite{Bohman} resolved Spencer's conjecture, showing
that indeed $\expec{e(\MM_n(K_3))} = \Theta(n^{3/2}(\ln n)^{1/2})$.
Bohman also proved a lower bound of $\Omega(n^{8/5}(\ln n)^{1/5})$ for
the expected number of edges in $\MM_n(K_4)$. In fact, Bohman's lower bounds
hold with probability that goes to $1$ as $n$ goes to $\infty$.
We discuss Bohman's argument and compare it to ours below.

As for upper bounds: The currently best upper bound on the expectation 
of $e(M_n(H))$, for $H$ that is strictly $2$-balanced over at least $4$
vertices is, by a result of Osthus and Taraz~\cite{OsthusT01}, at most
$O(n^{2-(v_H-2)/(e_H-1)}(\ln n)^{1/(\Delta_H-1)})$, where $\Delta_H$
denotes the maximum degree of $H$.

\subsection{Overview of the proof of Theorem~\ref{thm::main}}
\label{sec:1c}
Let $H$ be a regular, strictly $2$-balanced graph.  We would like
to analyse the random process generating $\MM_n(H)$. In order to do
this--and the reason will soon be apparent--it would be convenient
for us to think slightly differently about the definition of $\beta$.
Let $\GG(n,\rho)$ be the standard Erd\H{o}s-R\'{e}nyi random graph,
which is defined by keeping every edge of $K_n$ with probability $\rho$, 
independently of the other edges.
Then an alternative, equivalent
definition of $\beta$ is this: For every edge $f \in \GG(n,\rho)$
assign uniformly at random a birthtime $\beta(f) \in [0,\rho]$, and for
every edge $f \in K_n \setminus \GG(n,\rho)$ assign uniformly at random a
birthtime $\beta(f) \in (\rho,1]$.  Clearly, in this definition, every
edge $f \in K_n$ is assigned a uniformly random birthtime $\beta(f)
\in [0,1]$ and so this new definition is equivalent to the original
definition of $\beta$.  Note that $\GG(n,\rho)$ denotes here the set of
edges in $K_n$ whose birthtime is at most $\rho$.  The main advantage of
this new view of $\beta$ is that in order to analyse 
the event $\eventcnd{f \in \MM_n(H)}{\beta(f) < \rho'}$ for some $\rho'
\le \rho$, it is enough to consider only the distribution of the birthtimes
of edges of $\GG(n,\rho)$.  Hopefully, for our choice of $\rho$,
$\GG(n,\rho)$ will be structured enough so that we could take advantage
of the structures appearing in it and use them to find a non-trivial
lower bound on the probability of $\eventcnd{f \in \MM_n(H)}{\beta(f)
< \rho'}$.  This is the basic idea of the proof. We next describe,
\emph{informally}, what structures in $\GG(n,\rho)$ we hope to take
advantage of in order to prove Theorem~\ref{thm::main}.

For an edge $f \in K_n$, let $\Lambda(f,\rho)$ be the set of all $G
\subseteq \GG(n,\rho) \setminus \{f\}$ such that $G \cup \{f\}$ is
isomorphic to $H$.  Fix an edge $f \in K_n$ and let $\rho' \le \rho$.
Assume that the event $\event{\beta(f) < \rho'}$ occurs. Suppose now
that we want to estimate the probability of the event $\event{f \in
\MM_n(H)}$, which, by linearity of expectation, is essentially what we
need to do in order to prove Theorem~\ref{thm::main}.
We seek a sufficient condition for the event $\event{f \in \MM_n(H)}$.
One such trivial event is this: Say that \emph{$f$ survives-trivially}
if for every graph $G \in \Lambda(f,\rho)$ there exists an edge $g
\in G$ such that $\event{\beta(g) > \beta(f)}$ occurs. Clearly if
$f$ survives-trivially then we have $\event{f \in \MM_n(H)}$. We can
improve this simple sufficient condition as follows.  Say that an edge
\emph{$g$ doesn't survive} if there exists $G' \in \Lambda(g,\rho)$ such
that for every edge $g' \in G'$ we have $\event{\beta(g') < \beta(g)}$
and $g'$ survives-trivially.  Note that if $g$ doesn't survive then
$\event{g \notin \MM_n(H)}$ occurs. Now say that \emph{$f$ survives} if
for every graph $G \in \Lambda(f,\rho)$ there exists an edge $g \in G$
such that either $\event{\beta(g) > \beta(f)}$ or $g$ doesn't survive.
Then the event that $f$ survives implies $\event{f \in \MM_n(H)}$.

Observe that the event that $f$ survives was defined above using an
underlying tree-like structure of constant depth, in which the root is
$f$, the set of children of any non-leaf edge $g$ is $\Lambda(g,\rho)$
and for any $G \in \Lambda(g,\rho)$, the set of children of $G$ is
simply the set of edges in $G$.  Using the same idea as in the previous
paragraph, we could have defined the event that $f$ survives using an
underlying tree-like structure which is much deeper than the constant
depth tree-like structure that was used above.  Intuitively, the deeper
this tree-like structure is -- the better the chances are for $f$ to
survive. Therefore, we would be interested in defining the event that
$f$ survives using a rather deep underlying tree-like structure. We will
then be interested in lower bounding the probability that $f$ survives.

Now, in order to analyse the event that $f$ survives, it would be
useful if the underlying tree-like structure $T$ is \emph{good} in
the following sence: Every edge that appears in $T$ appears exactly
once\symbolfootnote[2]{In this informal discussion, we cannot hope
that $T$ would be good, since for example, $f$ appears as an edge in
some $G' \in \Lambda(g,\rho)$ for some $g \in G \in \Lambda(f,\rho)$. We
will define in Section~\ref{sec::2} the tree $T$ slightly differently,
so that this situation is avoided, while still maintaining that if $f$
survives then $\event{f \in \MM_n(H)}$ occurs.  Yet, for the purpose
of communicating the idea of the proof, it would be useful to assume
that $T$ could be good.}.  The advantage of $T$ being good is that
for many of the edges that appear in $T$, the events that these edges
survive or doesn't survive are pairwise independent. This property can
be used to analyse recursively the event that $f$ survives. Hence, it would
be very helpful if we can show that $T$ is good with high probability.
Showing this is a key ingredient of our proof.

Given the informal discussion above, the proof of Theorem~\ref{thm::main}
looks very roughly as follows. At the first part of the proof we consider
the graph $\GG(n,\rho)$ for a relatively large $\rho$, and show that
for a fixed edge $f \in K_n$, with probability that approaches $1$
as $n$ goes to $\infty$, we can associate with $f$ a tree $T$ which is
similar to the tree-like structure described above and which is both
good and deep.  Then, the second part has this structure: We assume
first that $\event{\beta(f) < \rho'}$ occurs for some suitably chosen
$\rho' \le \rho$. We also assume that the tree $T$ that is associated
with $f$ is good and deep, which occurs with high probability. Then,
we associate with $f$ and $T$ an event which is essentially the event
that $f$ survives, as described informally above, and argue that this
event implies $\event{f \in \MM_n(H)}$. Lastly, we give an explicit lower
bound on the probability of the 
event that we have associated with $f$ and $T$.  This will
give us a lower bound on the probability of $\event{f \in \MM_n(H)}$
conditioned on $\event{\beta(f) < \rho'}$. For our choice of $\rho'$,
this will imply Theorem~\ref{thm::main}.

\subsubsection{Comparison with previous work}
The \emph{basic} idea that we have outlined in the overview above
was used already by Erd\H{o}s, Suen and Winkler~\cite{ErdosSW95}
and by Spencer~\cite{Spencer0a} for the case $H=K_3$. (Their results
have been mentioned above.)  In~\cite{ErdosSW95}, the authors have
analyzed the event that an edge $f$ survives-trivially, as described
above, and considered \emph{implicitly} the graph $\GG(n,1)$. This
elementary argument gives a reasonable lower bound on the probability
of $\eventcnd{f \in \MM_n(K_3)}{\beta(f) < an^{-1/2}}$, for small
constant $a$ (e.g., $a=1$).  In~\cite{Spencer0a} the graph $\GG(n,1)$
was again considered implicitly, but a more general event -- essentially
the event that an edge $f$ survives, with an underlying tree-like structure 
of constant depth -- was analyzed; Using this, Spencer
was able to give a lower bound on the probability of $\eventcnd{f \in
\MM_n(K_3)}{\beta(f) < an^{-1/2}}$, for $a$ being arbitrary large,
but \emph{constant} independent of $n$.  As we have discussed above,
we consider \emph{explicitly} the graph $\GG(n,\rho)$ and we do that for
some suitably chosen $\rho < 1$.  This is the key to our improvement.  For
example, for the case of $H=K_3$, this enables us to give a non-trivial
lower bound on the probability of $\eventcnd{f \in \MM_n(K_3)}{\beta(f)
< an^{-1/2}}$, for $a = \floor{(\ln n)^{1/24}}$.  Moreover, our arguments
apply for every other regular, strictly $2$-balanced graph.

\subsubsection{Comparison with Bohman's argument}
As stated above, Bohman~\cite{Bohman} have proved stronger bounds than
those given in Theorem~\ref{thm::main}, for the case where $H \in \{K_3,
K_4\}$.  To do this, Bohman uses the differential equation method. The
basic argument, applied for the case $H = K_3$, can be described as
follows.
First, a collection of random variables that evolve throughout the
random process is introduced and tracked throughout the evolution of
$\MM_n(K_3)$.  This collection includes, for example, the random variable
$O_i$, which denotes the set of edges that have not yet been traversed by
the process, and which can be added to the current graph without forming
a triangle, after exactly $i$ edges have been added to the evolving graph.
Now, at certain times during the process (i.e., at those times in which
new edges are added to the evolving graph), Bohman expresses the expected
change in the values of the random variables in the collection, using
the same set of random variables. This allows one to express the random
variables in the collection using the solution to an autonomous system
of ordinary differential equations.
The main technical effort in Bohman's work
then shows that the random variables in the collection are tightly concentrated
around the trajectory given by the solution to this system.
The particular solution to the system then implies that with high
probability $O_I$ is still large for $I := n^{3/2}(\ln n)^{1/2}/32$. 
This gives Bohman's lower bound on the expected number of edges in $\MM_n(K_3)$.

We remark that Bohman's argument probably can be used to analyse the
random process generating $\MM_n(H)$ for $H \notin \{K_3, K_4\}$, and
this can most likely lead to stronger lower bounds than those given
in Theorem~\ref{thm::main}.  In comparison with Bohman's argument, our
argument is more direct in the sence that it considers a single edge and
estimates directly the probability of it being included in $\MM_n(H)$.
We remark that our argument can be strengthened and
generalized in the following
way for the case $H=K_3$. One can use our basic argument so as to give
an asymptotically tight expression for the probability that a fixed 
triangle-free graph $F$ is included in $\MM_n(K_3)$, conditioned on the
event that the edges of $F$ all have birthtimes which are relatively, but
not trivially small.  This, in turn, can be used to tackle the following 
question, which is left open even after Bohman's breakthrough.
Suppose we trim the random process generating $\MM_n(K_3)$ right after
every edge whose birthtime is less than $cn^{-1/2}$ has been traversed,
where $c = (\ln n)^{1/24}$. That is, let us consider the \emph{trimmed
graph} $\MM_n(K_3) \cap \{f : \beta(f) < cn^{-1/2}\}$.  We may ask what
is the number of paths of length $2$ in the trimmed graph.  Bohman's
argument does not answer this question, but rather places an upper bound
of $\binom{n}{2} \cdot (\ln n)^2$ on that number.  Yet, the above-mentioned 
strengthening and generalization together with the second moment method 
can be used to show that the number of paths of length $2$ in the trimmed 
graph is concentrated around $\binom{n}{2} \cdot \ln c$. Similarly, one 
can prove concentration results for the number of small cycles in the 
trimmed graph.  
%

\subsection{Organization of the paper}
In Section~\ref{sec::2} we give the basic definitions we use throughout
the paper and in particular, we give the formal definition of what we
have referred to above as a good tree-like structure. We also state in
Section~\ref{sec::2} the two main lemmas we prove throughout the paper
and argue that these lemmas imply the validity of Theorem~\ref{thm::main}.
The two main lemmas are proved in Sections~\ref{sec::3}~and~\ref{sec::4}
and these two sections correspond to the two parts of the proof that
were sketched at Section~\ref{sec:1c}.

\subsection{Basic notation and conventions}
We use $K_n$ to denote the complete graph over the vertex set $[n] :=
\{1,2,\ldots,n\}$. We set $[0] := \emptyset$. We use $f,g,g'$ to denote
edges of $K_n$ and $F,G,G'$ to denote subgraphs of $K_n$ or subgraphs
of any other fixed graph.
Throughout the paper, the hidden constants in the big-O and big-Omega
notation, are either absolute constants or depend only on an underlying
fixed graph $H$ which should be understood from the context.  If $x=x(n)$
and $y=y(n)$ are functions of $n$, we write $y = o(x)$ if $y/x$ goes to
$0$ as $n$ goes to $\infty$ and $y = \omega(x)$ if $y/x$ goes to
$\infty$ as $n$ goes to $\infty$.
%

\section{Main lemmas and proof of Theorem~\ref{thm::main}}
\label{sec::2}
In this section we give the overall structure of the proof of
Theorem~\ref{thm::main}, including the required basic definitions and
two key lemmas--whose validity imply the theorem. We fix once and for
the rest of this paper a regular, strictly $2$-balanced graph $H$ and
prove Theorem~\ref{thm::main} for that specific $H$.  We always think 
of $n$ as being sufficiently large, and define the following functions of $n$.
\begin{definition}
\label{def::kk}
Define
\begin{eqnarray*}
k &=& k(n) \,\,\quad :=\quad  n^{(\ln n)^{-1/2}}, \\
\rho  &=& \rho(n) \,\,\quad :=\quad  kn^{-(v_H-2)/(e_H-1)}, \\
c &=& c(n) \,\,\, \quad :=\quad  \floor{(\ln n)^{1/(8e_H)}}, \qquad\textrm{and} \\
D &=& D(n) \quad :=\quad  2\floor{(\ln n)^{1/4}}+1. 
\end{eqnarray*}
\end{definition}
In order to prove Theorem~\ref{thm::main}, we will show that for our
fixed graph $H$, and for every edge $f \in K_n$,
\begin{eqnarray}
\label{eq::main}
\Probcnd{f \in \MM_n(H)}{\beta(f) < cn^{-(v_H-2)/(e_H-1)}} =
  \Omega\bigg(\frac{(\ln c)^{1/(e_H-1)}}{c}\bigg).
\end{eqnarray}
Note that~(\ref{eq::main}) implies Theorem~\ref{thm::main}: Since
$\prob{\beta(f) < cn^{-(v_H-2)/(e_H-1)}} = cn^{-(v_H-2)/(e_H-1)}$, it
follows from~(\ref{eq::main}) that for every $f \in K_n$, $\prob{f \in
\MM_n(H)} = \Omega(n^{-(v_H-2)/(e_H-1)} (\ln c)^{1/(e_H-1)})$. Using the
fact that $\ln c = \Omega(\ln\ln n)$ and using linearity of expectation,
this last bound implies Theorem~\ref{thm::main}.  It thus remains to
prove~(\ref{eq::main}).  The rest of this section is devoted to outlining
the proof of~(\ref{eq::main}).

Recall that for an edge $f \in K_n$, we define $\Lambda(f,\rho)$ to be
the set of all $G \subseteq \GG(n,\rho) \setminus \{f\}$ such that $G
\cup \{f\}$ is isomorphic to $H$.  We now set up to define what we have
referred to in the introduction as a good tree-like structure.

A rooted tree $T$ is a directed tree with a distinguished node, called
the \emph{root}, which is connected by a directed path to any other
node in $T$. If $u$ is a node in $T$ then the set of nodes that are
adjacent\symbolfootnote[3]{We say that node $v$ is adjacent to node $u$
in a given directed graph, if there is a directed edge from $u$ to $v$.}
to $u$ in $T$ is denoted by $\Gamma_T(u)$. The height of a node $u$
in a rooted tree $T$ is the length of the longest path from $u$ to a
leaf. The height of a rooted tree is the height of its root.  We shall
consider labeled (rooted) trees. If $u$ is a node in a labeled tree $T$,
we denote by $L_T(u)$ the label of the node $u$ in $T$.

\begin{definition}
[$T_{f,d}$]
\label{def::afdtree}
Let $f \in K_n$ and $d \in \N$.  We define inductively a labeled,
rooted tree $T_{f,d}$ of height $2d$.  The nodes at even distance from
the root will be labeled with edges of $K_n$.  The nodes at odd distance
from the root will be labeled with subgraphs of $K_n$.
\begin{itemize}
\item $T_{f,1}$: 
\begin{itemize}
\item The root $v_0$ of $T_{f,1}$ is labeled with the edge $f$.
\item For every subgraph $G_1 \in \Lambda(f,\rho)$: Set a new node $u_1$
which is adjacent to $v_0$ and whose label is $G_1$;  Furthermore, for
each edge $g \in G_1$ set a new node $v_1$ which is adjacent to $u_1$
and whose label is $g$.
\end{itemize}
\item $T_{f,d}$, $d \ge 2$: 
We construct the tree $T_{f,d}$ by adding new nodes to $T=T_{f,d-1}$ as
follows.  Let $(v_0,u_1,v_1,\ldots,u_{d-1},v_{d-1})$ be a directed path
in $T_{f,d-1}$ from the root $v_0$ to a leaf $v_{d-1}$.  Let $g_{d-1}
= L_T(v_{d-1})$ and $g_{d-2} = L_T(v_{d-2})$.  For every subgraph $G_d
\in \Lambda(g_{d-1},\rho)$ such that $g_{d-2} \notin G_d$ do: Set a
new node $u_d$ which is adjacent to $v_{d-1}$ and whose label is $G_d$;
Furthermore, for each edge $g_d \in G_d$ set a new node $v_d$ which is
adjacent to $u_d$ and whose label is $g_d$.
\end{itemize}
\end{definition}

\begin{definition}
[good tree]
\label{def::goodtree}
Let $f \in K_n$ and $d \in \N$. Consider the tree $T=T_{f,d}$ and let
$v_0$ denote the root of $T$. We say that $T$ is \emph{good} if the
following three properties hold:
\begin{itemize}
\item[P1] If $G$ is the label of a node $u$ at odd distance from $v_0$
then $G \cap \{f\} = \emptyset$.
\item[P2] If $G,G'$ are the labels of two distinct nodes at odd distance
from $v_0$ then $G \cap G' = \emptyset$.
\item[P3] If $g$ is the label of a non-leaf node $v$ at even distance
from $v_0$ then $|\Gamma_T(v)| = |\Lambda(g,\rho)| - O(1)$.
\end{itemize}
\end{definition}

Recall the definition of $\rho$ and note that the expected size
of $\Lambda(g,\rho)$ is $\lambda k^{e_H-1}$, where $\lambda =
\lambda'(1-o(1))$ and $\lambda' \le 1$ depends only on $H$.
(This follows from the fact that for every edge $g \in K_n$, the
cardinality of $\Lambda(g,1)$ is between $\binom{n-2}{v_H-2}$ and
$(v_H-2)! \binom{n-2}{v_H-2}$, and from the fact that for every $G
\in \Lambda(g,1)$, the probability of $\event{G \in \Lambda(g,\rho)}$
is $\rho^{e_H-1}$.)  Define the event $E_1$ to be the event that for
every edge $g \in K_n$,
\begin{displaymath}
\lambda k^{e_H-1} - k^{e_H/2-1/3}/2 \le
|\Lambda(g,\rho)| \le \lambda k^{e_H-1} + k^{e_H/2-1/3}/2.
\end{displaymath}
For an edge $f \in K_n$, let $E_2(f)$ be the event that $T_{f,D}$ is good.
The next lemma is proved in Section~\ref{sec::3}.
\begin{lemma}
\label{lemma::toprove1}
For every edge $f \in K_n$, 
\begin{displaymath}
\prob{E_2(f) \cap E_1} = 1-o(1).
\end{displaymath}
\end{lemma}
Assuming that the event $E_2(f) \cap E_1$ occurs, the tree $T_{f,D}$
is exactly what we have
referred to informally in the introduction as a good tree-like
structure. Assuming that such a structure exists in $\GG(n,\rho)$,
we derive in Section~\ref{sec::4} a lower bound on the probability
of $\event{f \in \MM_n(H)}$, conditioned on $\event{\beta(f) <
cn^{-(v_H-2)/(e_H-1)}}$. Formally, we prove the next lemma.
\begin{lemma}
\label{lemma::toprove2}
For every edge $f \in K_n$,
\begin{eqnarray*}
\Probcnd{f \in \MM_n(H)}{E_2(f) \cap E_1,
\beta(f) < cn^{-(v_H-2)/(e_H-1)}}
= \Omega\bigg(\frac{(\ln c)^{1/(e_H-1)}}{c}\bigg).
\end{eqnarray*}
\end{lemma}
Trivially\footnote{See the errata at the end of the paper.}, Lemmas~\ref{lemma::toprove1}~and~\ref{lemma::toprove2}
imply~(\ref{eq::main}) and hence Theorem~\ref{thm::main}.  Therefore,
in order to prove Theorem~\ref{thm::main}, it remains to prove these
two lemmas.

\section{Proof of Lemma~\ref{lemma::toprove1}}
\label{sec::3}
The proof is divided to two parts. In the first part, given at
Section~\ref{sec::31}, we lower bound the probability of the event
$E_1$. In the second part we lower bound the probability of the event
$E_2(f)$. Since these two lower bounds would be shown to be $1-o(1)$,
Lemma~\ref{lemma::toprove1} will follow.

\subsection{Bounding $\prob{E_1}$}
\label{sec::31}
In this subsection we show that the probability of the event $E_1$ is
$1-o(1)$.  In order to do this, since there are at most $n^2$ edges in
$K_n$, it suffices to fix an edge $g \in K_n$ and show that the following
two equalities hold:
\begin{eqnarray}
\label{eq::treea}
\Prob{|\Lambda(g,\rho)| \ge \lambda k^{e_H-1} - k^{e_H/2-1/3}/2} =
   1-n^{-\omega(1)},\\
\label{eq::treeb}
\Prob{|\Lambda(g,\rho)| \le \lambda k^{e_H-1} + k^{e_H/2-1/3}/2} = 
   1-n^{-\omega(1)}.
\end{eqnarray}
Throughout this section we will use several times the following fact.
\begin{fact}
\label{fact::ufact}
There exists a constant $\epsilon_H > 0$, that depends only on $H$, such
that the following holds for all sufficiently large $n$: 
If $F \subsetneq H$ and $v_F \ge 3$ then 
\begin{eqnarray*}
n^{v_H-v_F}\rho^{e_H-e_F} \le n^{-\epsilon_H}.
\end{eqnarray*}
\end{fact}
\begin{proof}
Fix $F \subsetneq H$ with $v_F \ge 3$.  Since $H$ is strictly $2$-balanced,
we have that $(e_F-1)(v_H-2)/(e_H-1) <
v_F-2$. Hence, there exists a constant $\epsilon_H'>0$ such that
$n^{-v_F+2}\rho^{-e_F+1}  =  n^{-v_F+2 + (e_F-1)(v_H-2)/(e_H-1)+o(1)}
\le n^{-\epsilon_H'+o(1)}$ (here we have also used the fact that
$k = n^{o(1)}$).  We also
note that $n^{v_H-2}\rho^{e_H-1} = k^{e_H-1} = n^{o(1)}$.  Therefore,
$n^{v_H-v_F} \rho^{e_H-e_F} = n^{v_H-2-v_F+2}\rho^{e_H-1-e_F+1} \le
n^{-\epsilon_H'+o(1)}$. To complete the proof, take the subgraph $F
\subsetneq H$ with $v_F \ge 3$ which minimizes $\epsilon_H'$ above,
and for that particular $\epsilon_H'$, take $\epsilon_H = \epsilon_H'/2$.
\end{proof}
We prove~(\ref{eq::treea})~and~(\ref{eq::treeb}) in
Sections~\ref{sec::lowertail}~and~\ref{sec::uppertail}, respectively.

\subsubsection{The lower tail}
\label{sec::lowertail}
For $G \in \Lambda(g,1)$, let $X_G$ be the indicator random
variable for the event $\event{G \subseteq \GG(n,\rho)}$.  Let $X =
\sum_{G\in\Lambda(g,1)} X_G$.  Then $|\Lambda(g,\rho)| = X$ and $\expec{X}
= \lambda k^{e_H-1}$.  Let $\Delta = \sum_{G, G'} \expec{X_G \cap X_{G'}}$,
where the sum ranges over all ordered pairs $G, G' \in \Lambda(g,1)$ with
$G \cap G' \ne \emptyset$ (this includes the pairs $G, G'$ with $G=G'$).
Then from Janson~\cite{DBLP:journals/rsa/Janson90a} we have that for
every $0 \le t \le \expec{X}$,
\begin{eqnarray}
\label{eq::janson}
\prob{X \le \expec{X} - t} \le \exp\Big(-\frac{t^2}{2\Delta}\Big).
\end{eqnarray}
We now bound $\Delta$ from above.  In order to do this, first note
that for every $F \subseteq H$ and for every $G \in \Lambda(g,1)$,
the number of subgraphs $G' \in\Lambda(g,1)$ such that $(G\cup\{g\})
\cap (G'\cup\{g\})$ is isomorphic to $F$ is at most $O(n^{v_H-v_F})$.
Also, the number of subgraphs $G \in \Lambda(g,1)$ is trivially at most
$n^{v_H-2}$. Hence, denoting by $\sum_{F}$ the sum over all $F \subseteq
H$ with $v_F \ge 3$, the number of pairs $G, G'$ which contribute to
$\Delta$ is at most $\sum_{F} O(n^{2v_H-v_F-2})$. For every pair $G, G'$
as above, if $(G \cup \{g\}) \cap (G' \cup \{g\})$ is isomorphic to $F$
then $\expec{X_G \cap X_{G'}} = \rho^{2e_H -e_F -1}$. Hence
\begin{eqnarray}
\nonumber
\Delta \le \sum_{F} O(n^{2v_H -v_F -2} \rho^{2e_H -e_F -1})
&=& \sum_{F} O(n^{2v_H -4 -v_F + 2} \rho^{2e_H -2 -e_F + 1})  \\
\label{eq::laso}
&=& k^{2(e_H-1)} \sum_{F} O(n^{-v_F + 2} \rho^{-e_F + 1}). 
\end{eqnarray}
Now if $F \subsetneq H$ and $v_F \ge 3$ then by the fact that 
$H$ is strictly $2$-balanced we have $n^{-v_F + 2} \rho^{-e_F + 1}
\le n^{-\epsilon'_H+o(1)}$ for some $\epsilon'_H>0$ that depends
only on $H$ (see the proof of Fact~\ref{fact::ufact}).  If $F$ on
the other hand satisfies $F = H$, then $n^{-v_F + 2} \rho^{-e_F + 1}
= k^{-(e_H-1)}$. Hence, we can further upper bound~(\ref{eq::laso})
by $O(k^{e_H-1})$. This upper bound on $\Delta$
can be used with~(\ref{eq::janson}) to show that
\begin{eqnarray*}
\prob{X \ge \expec{X} - k^{e_H/2-1/3}/2} \ge 1 - \exp\Big(-\Omega\big(k^{1/3}\big)\Big).
\end{eqnarray*}
This gives us~(\ref{eq::treea}).

\subsubsection{The upper tail}
\label{sec::uppertail}
We are interested in giving a lower bound on the probability of the event
that $|\Lambda(g,\rho)| \le \lambda k^{e_H-1} + k^{e_H/2-1/3}/2$.  The
technique we use is due to Spencer~\cite{DBLP:journals/jct/Spencer90a}.
Let $\G$ be the graph over the vertex set $\Lambda(g,\rho)$ and
whose edge set consists of all pairs of distinct vertices $G,G' \in
\Lambda(g,\rho)$ such that $G \cap G' \ne \emptyset$.  Let $W_1$ be the
size of the maximum independent set in $\G$. Let $W_2$ be the size
of the maximum induced matching in $\G$.  Let $W_3$ be the maximum
degree of $\G$. Then by a simple argument, one gets that the number
of vertices in $\G$, which is $|\Lambda(g,\rho)|$, is at most $W_1 +
2W_2W_3$. (Indeed, we can partition the set of vertices of $\G$
to those that are adjacent to a vertex in some fixed induced matching
of largest size, and to those that are not. The first part of the
partition trivially has size at most $2W_2W_3$. The second part of the
partition is an independent set and so has size at most $W_1$.) Hence,
in order to prove~(\ref{eq::treeb}), it is enough to show that $W_1$
and $W_2W_3$ are sufficiently small with probability $1-n^{-\omega(1)}$.
Specifically we will show the following:
\begin{eqnarray}
\label{eq::i1}
\prob{W_1 \ge \lambda k^{e_H-1} + k^{e_H/2-1/3}/3} &\le& n^{-\omega(1)}, \\
\label{eq::i2}
\prob{W_2 \ge \ln n} &\le& n^{-\omega(1)}, \\
\label{eq::i3}
\prob{W_3 \ge \ln n} &\le& n^{-\omega(1)}.
\end{eqnarray}
Note that by the argument above,~(\ref{eq::i1}--\ref{eq::i3})
imply via the union bound that with probability $1-n^{-\omega(1)}$,
$|\Lambda(g,\rho)| \le \lambda k^{e_H-1} + k^{e_H/2-1/3}/2$, so it remains
to prove~(\ref{eq::i1}--\ref{eq::i3}).

We start by proving~(\ref{eq::i3}). Since there are at most $n^{v_H-2}$
subgraphs in $\Lambda(g,1)$, it is enough to fix $G \in \Lambda(g,1)$
and prove that, with probability $1-n^{-\omega(1)}$, either $G$ is not
a vertex in $\G$, or $G$ has degree less than $\ln n$ in $\G$.
So let us fix $G \in \Lambda(g,1)$.  For $t \ge 0$, we say that a
sequence $S=(G_j)_{j=0}^{t}$ of subgraphs $G_j \in \Lambda(g,1)$ is
a \emph{$(G,t)$-star}, if $G_0 = G$ and if for every $j \ge 1$ the
following two conditions hold: (i)~$G_0 \cap G_j \ne \emptyset$, and
(ii)~$G_j$ has an edge which do not belong to any $G_{j'}$, $j'<j$. We
say that $\GG(n,\rho)$ contains a $(G,t)$-star $S$ and write $\event{S
\subseteq \GG(n,\rho)}$ for that event, if for every subgraph $G_j \in S$,
$G_j \subseteq \GG(n,\rho)$.  We first observe that if no $(G,t)$-star
is contained in $\GG(n,\rho)$, then either $G$ is not a vertex of $\G$,
or the degree of $G$ in $\G$ is at most $O(t^{e_H})$.  Indeed, if $t=0$
then clearly $G$ is not a vertex in $\G$; So assume $t \ge 1$ and and
let $S$ be a maximal $(G,t')$-star that is contained in $\GG(n,\rho)$
(here maximal means that $\GG(n,\rho)$ contains no $(G,t'+1)$-star). Then
by maximality of $S$, any vertex that is adjacent to $G$ in $\G$ is
either in the sequence $S$, or is fully contained in $E(S)$, where $E(S)$
denotes the set of all edges of the subgraphs in $S$.  Since $|E(S)| =
O(t)$, it then follows trivially that the number of vertices
adjacent to $G$ in $\G$ is at most $O(t^{e_H})$.  Hence, in order
to prove~(\ref{eq::i3}) it remains to show that with probability
$1-n^{-\omega(1)}$, $\GG(n,\rho)$ contains no $(G,\floor{\ln\ln
n})$-star, say. For brevity, below we assume that $\ln\ln n$ is
an integer.

\def\Star{\textrm{Star}}

Let $Z_t$ denote the number of $(G,t)$-stars that are contained in
$\GG(n,\rho)$, where $G$ is the subgraph fixed above.  Since the
probability that $\GG(n,\rho)$ contains a $(G,t)$-star is at
most $\expec{Z_t}$, it is enough to show that for $t = \ln\ln
n$, $\expec{Z_t}$ is upper bounded by $n^{-\omega(1)}$. Denote by
$\Star_t$ the set of all $(G,t)$-stars.  For $S = (G_j)_{j=0}^{t-1} \in
\Star_{t-1}$, denote by $\E_t(S)$ the set of all $G_t \in \Lambda(g,1)$
such that $(S,G_t) := (G_j)_{j=0}^t \in \Star_t$. Then for $t \ge 1$,
\begin{eqnarray*}
\expec{Z_t} &=& \sum_{S \in \Star_t} \prob{S \subseteq \GG(n,\rho)}  \\
            &=& 
              \sum_{S \in \Star_{t-1}} \prob{S \subseteq \GG(n,\rho)} \cdot
	      \sum_{G_t \in \E_t(S)} 
	      \probcnd{G_t \subseteq \GG(n,\rho)}{S \subseteq \GG(n,\rho)}.
\end{eqnarray*}
Take $t \in [\ln\ln n]$ and fix $S \in \Star_{t-1}$.  Note that
the number of subgraphs $G_t \in \E_t(S)$ such that $(G_t \cup
\{g\}) \cap (E(S) \cup \{g\})$ is isomorphic to $F \subseteq H$
is at most $O(n^{v_H-v_F}t^{v_H-2})$, which for our choice of $t$
is at most $n^{v_H-v_F+o(1)}$. Moreover, for such subgraphs $G_t$,
we have that the probability of $\eventcnd{G_t \subseteq \GG(n,\rho)}{S
\subseteq \GG(n,\rho)}$ is exactly $\rho^{e_H-e_F}$.  Also note that
for every $G_t \in \E_t(S)$, $(G_t \cup \{g\}) \cap (E(S) \cup \{g\})$
is isomorphic to some $F \subsetneq H$ with $v_F \ge 3$.  Hence,
letting $\sum_F$ be the sum over all $F \subsetneq H$ with $v_F \ge 3$,
we have for our choice of $t$ that there exists $\epsilon_H > 0$ such that:
\begin{displaymath}
\sum_{G_t \in \E_t(S)} 
\probcnd{G_t \subseteq \GG(n,\rho)}{S \subseteq \GG(n,\rho)} \le
\sum_{F} O(n^{v_H - v_F + o(1)} \rho^{e_H - e_F}) \le n^{-\epsilon_H+o(1)},
\end{displaymath}
where the last inequality is from Fact~\ref{fact::ufact}. Hence for $t
\in [\ln\ln n]$,
\begin{eqnarray}
\label{eq::lili}
\expec{Z_t} \le \expec{Z_{t-1}} \cdot n^{-\epsilon_H+o(1)}.
\end{eqnarray}
As there is only one $(G,0)$-star, $\expec{Z_0} \le 1$. Hence we conclude
from~(\ref{eq::lili}) that $\expec{Z_t} \le n^{-(\epsilon_H-o(1))t}$
for all $t \in [\ln\ln n]$. Thus, for $t = \ln\ln n$, $\expec{Z_t}
= n^{-\omega(1)}$. This concludes the proof of~(\ref{eq::i3}).

\def\Match{\textrm{Match}}

Next we prove~(\ref{eq::i2}).  Let $Y_t$ denote the number of induced matchings
of size $t$ in $\G$. Since the expectation of $Y_t$ is an upper bound
on the probability that there exists an induced matching of size $t$
in $\G$, in order to prove~(\ref{eq::i2}) it is enough to show that
for $t = \floor{\ln n}$, $\expec{Y_t} = n^{-\omega(1)}$.
Let $\G^*$ be the graph whose vertex set is $\Lambda(g,1)$ and whose
edge set consists of all pairs of distinct vertices $G,G'$ such that
$G \cap G'\ne \emptyset$.  Let $\Match_t$ be the collection of all
induced matchings of size $t$ in $\G^*$.  For $M \in \Match_{t-1}$,
let $\E_t(M)$ denote the set of all edges $(G_t,G_t')$ in $\G^*$
such that $M \cup \{(G_t,G_t')\} \in \Match_t$.  The number of edges
$(G_t,G_t') \in \E_t(M)$ such that $(G_t \cup \{g\}) \cap (G_t' \cup
\{g\})$ is isomorphic to $F \subseteq H$ is at most $O(n^{2v_H - 2 -
v_F})$;  Moreover, for such an edge $(G_t,G_t')$, the probability of
the event $\event{G_t, G_t' \subseteq \GG(n,\rho)}$ is $\rho^{2e_H - 1 -
e_F}$, even conditioning on the event $\event{G,G' \subseteq \GG(n,\rho):
(G,G') \in M}$.  Trivially, for an edge $(G_t,G_t') \in \E_t(M)$, we have
that $(G_t \cup \{g\}) \cap (G_t' \cup \{g\})$ is isomorphic to a proper
subgraph $F$ of $H$ over at least $3$ vertices.  Thus, if $\sum_F$ is
the sum over all $F \subsetneq H$ with $v_F \ge 3$, we have for $t \ge 1$,
\begin{eqnarray}
\label{eq:dis}
\nonumber
\expec{Y_t} &=& \sum_{M \in \Match_t} \Prob{G,G' \subseteq \GG(n,\rho) : (G,G') \in M} \\
\nonumber
&\le&\sum_{M \in \Match_{t-1}} \Prob{G,G' \subseteq \GG(n,\rho) : (G,G') \in M} \cdot \\
\nonumber
&& \sum_{(G_t,G_t') \in \E_t(M)} \Probcnd{G_t,G_t' \subseteq \GG(n,\rho)}
                                         {G,G' \subseteq \GG(n,\rho) : (G,G') \in M} \\
&\le& \expec{Y_{t-1}} \cdot \sum_{F} O(n^{2v_H-2-v_F} \rho^{2e_H-1-e_F})
\,\,\,\le\,\,\, \expec{Y_{t-1}} \cdot n^{-\epsilon_H+o(1)},
\end{eqnarray}
where the last inequality follows from the fact that
$n^{v_H-2}\rho^{e_H-1} = k^{e_H-1} = n^{o(1)}$ and from
Fact~\ref{fact::ufact}, so $\epsilon_H > 0$ depends only on $H$.
Since trivially $\expec{Y_0} = 1$, from~(\ref{eq:dis}) we can conclude
that $\expec{Y_t} = n^{-\omega(1)}$, for $t = \floor{\ln n}$.
This gives us~(\ref{eq::i2}).

Lastly, we prove~(\ref{eq::i1}).  For this we use the next tail
bound due to Spencer~\cite{DBLP:journals/jct/Spencer90a}~(See
also~\cite{janson00random}, Lemma~2.46).  If $X$ denotes the number of
vertices in $\G$ then
\begin{eqnarray}
\label{eq::jans}
\prob{W_1 \ge \expec{X} + t} \le \exp\bigg(-\frac{t^2}{2(\expec{X}+t/3)}\bigg).
\end{eqnarray}
Using the fact that $\expec{X} = \lambda k^{e_H-1}$, taking $t =
k^{e_H/2-1/3}/3$, using the fact that $k^{1/3} = \omega(\ln n)$, we can
conclude from~(\ref{eq::jans}) the validity of~(\ref{eq::i1}).

\subsection{Bounding $\prob{E_2(f)}$}
\label{sec::32}
For the rest of this section we fix an edge $f \in K_n$.  We show
that $E_2(f)$ occurs with probability $1-o(1)$. 
\begin{definition}
[bad sequence]
\label{def::bad}
Let $S = (G_1,G_2,\ldots,G_d)$ be a sequence of subgraphs of $K_n$
with $2 \le d \le 2D$.  We say that $S$ is a \emph{bad
sequence} if the following three items hold simultaneously:
\begin{enumerate}
\item For all $j \in [d]$, $G_j \in \Lambda(g, 1)$ for some edge
$g \in \{f\} \cup \bigcup_{i<j} G_i$.
\item For all $j \in [d-1]$, $G_j$ shares exactly $2$ vertices and $0$ edges
with $\{f\} \cup \bigcup_{i<j} G_i$.
\item $G_d$ shares at least $3$ vertices and at most $e_H-2$ edges
with $\{f\} \cup \bigcup_{i<d} G_i$.
\end{enumerate}
\end{definition}

For a bad sequence $S = (G_1,G_2,\ldots,G_d)$, write $\event{S \subseteq
\GG(n,\rho)}$ for the event that for every $j \in [d]$, $\event{G_j
\subseteq \GG(n,\rho)}$ occurs.  Let $E_3$ be the event that for all
bad sequences $S$, $\event{S \subseteq \GG(n,\rho)}$ does not occur. The
next two propositions imply the required lower bound of $1-o(1)$ on the
probability of $E_2(f)$, by first showing that $E_3$ implies $E_2(f)$
and then showing that the probability of $E_3$ is $1-o(1)$.
\begin{proposition}
\label{prop:lkj}
$E_3$ implies $E_2(f)$.
\end{proposition}
\begin{proof}
Assume $E_3$ occurs. Then for every bad sequence $S$, $\event{S
\subseteq \GG(n,\rho)}$ does not occur.  To prove the assertion in
the proposition, we need to show that the tree $T_{f,D}$ defined in
Definition~\ref{def::afdtree} is a good tree.  To do this, we need to
show that $T_{f,D}$ satisfies properties P1, P2 and P3, as given in
Definition~\ref{def::goodtree}.  We start the proof by showing, using
the following claim, the $T_{f,D}$ satisfies property~P1 (and part
of property~P2).

\begin{claim}
\label{claim::clm0}
For $d \in [D]$,
let $P=(v_0,u_1,v_1,\ldots,u_{d-1},v_{d-1},u_d,v_d)$ be a directed
path in $T_{f,d}$ from the root $v_0$ to a leaf $v_d$.  Let $G_j$
be the label of node $u_j$ and let $g_j$ be the label of node $v_j$
(so that $g_0 = f$).  Then (i) $G_d \cap \{f\} = \emptyset$, and (ii)
$G_d \cap G_i = \emptyset$, for every $0 \le i \le d-1$.
\end{claim}
\begin{proof}
The proof is by induction on $d$. Clearly, the claim is valid for $d=1$,
as by definition, any subgraph in $\Lambda(f,\rho)$ does not contain
the edge $f$. Let $d \ge 2$, $d \in [D]$ and assume the claim holds
for $d-1$.  We prove the claim for $d$.  Let $S = (G_1,G_2,\ldots,G_d)$
be the sequence of the labels of the nodes $u_i$, $i \in [d]$, along the
path $P$.  Assume for the sake of contradiction that $G_d$ shares some edge
with $\{f\} \cup \bigcup_{i<d} G_i$.  We shall reach a contradiction by 
showing that either $S$ is a bad sequence (this contradicts the occurrence
of $E_3$), or $P$ is not a directed path in $T_{f,d}$.  

Note first that from the induction hypothesis we have that
for every $j \in [d-1]$, $G_j$ shares no edge with 
$\{f\} \cup \bigcup_{i<j} G_i$. We claim that this implies also that
for every $j \in [d-1]$, $G_j$ shares exactly $2$ vertices with
$\{f\} \cup \bigcup_{i<j} G_i$. Indeed, for $d=2$ this claim 
is true by definition. If the claim is not true for $d \ge 3$ then
we have for some $j \in [d-1]$, $j \ge 2$, that $(G_1, G_2, \ldots, G_j)$ 
is a bad sequence, contradicting $E_3$.

Now, by assumption, $G_d$ shares some edge with
$\{f\} \cup \bigcup_{i<d} G_i$. If we also have that
$G_d$ shares at most $e_H-2$ edges with 
$\{f\} \cup \bigcup_{i<d} G_i$ then by the observation 
made in the previous paragraph we are done, since this 
implies that $S$ is a bad sequence.  Therefore, we can 
assume for the rest of the proof that $G_d$ shares all 
of its $e_H-1$ edges with $\{f\} \cup \bigcup_{i<d} G_i$.
We shall reach a contradiction by showing that $P$ is not
a directed path in $T_{f,d}$.

Write $g_{d-1} = \{a,b\}$ and $g_{d-2} = \{x,y\}$ and recall that
$G_{d-1} \in \Lambda(g_{d-2},\rho)$ and $G_d \in \Lambda(g_{d-1},\rho)$.
Observe that $g_{d-2} \ne g_{d-1}$. Hence, we may assume without loss
of generality that $a \notin \{x,y\}$. Note that $a$ is a vertex of
both $G_d$ and $G_{d-1}$.  Now, a \emph{key observation} is that any edge
in $G_d$ that is adjacent to $a$ must belong also to $G_{d-1}$,
for otherwise, the subgraph $G_{d-1}$ will share $3$ vertices
($x, y$ and $a$) with $\{f\} \cup \bigcup_{i<d-1} G_i$--and that 
contradicts the fact established above.
More generally and for the same reason, if $a' \notin \{x, y\}$
is a vertex of both $G_d$ and $G_{d-1}$, then any edge adjacent to $a'$
in $G_d$ must also belong to $G_{d-1}$.  With that key observation at hand,
we conclude the proof by reaching a contradiction for every possible
choice for the graph $H$.

Suppose first that $H = K_3$. Without loss of generality, we have $b =
x$. Now, since any edge that is adjacent to $a$ in $G_d$ must also be
an edge in $G_{d-1}$, it follows that $\{a, y\}$ is an edge in $G_d$.
Therefore, $\{a, b=x, y\}$ is the set of vertices of $G_d$ and so $g_{d-2}
= \{x, y\}$ is an edge in $G_d$. But, by Definition~\ref{def::afdtree},
this contradicts the assumption that $P$ is a directed path in $T_{f,d}$.

To reach a contradiction for other regular, strictly $2$-balanced graphs,
we need the following fact.
\begin{fact}
\label{fact:f1}
Let $H$ be a regular, strictly $2$-balanced graph with $v_H \ge 4$.  Let $\{x,
y\}$ be an edge in $H$ and let $a, a' \notin \{x, y\}$ be two distinct
vertices in $H$. Then there is a path from $a$ to $a'$ in $H$ that avoids
the vertices $\{x, y\}$.
\end{fact}
\begin{proof}
Assume for the sake of contradiction that there exist two distinct
vertices $a, a' \notin \{x, y\}$ such that every path in $H$ from
$a$ to $a'$, if there exists any, must go through a vertex in $\{x,
y\}$. This implies that we can write $H$ as the union of two graphs,
$H_1$ and $H_2$, that share only the edge $\{x, y\}$ and the vertices
$x$ and $y$, and such that $a$ is a vertex in $H_1$ and $a'$ is a 
vertex in $H_2$. Note that $v_{H_1}, v_{H_2} \ge 3$ and that $H_1$ 
and $H_2$ are proper subgraphs of $H$.  Without loss of generality,
we assume that $(e_{H_1}-1)/(v_{H_1}-2) \ge (e_{H_2}-1)/(v_{H_2}-2)$.
Now, by the fact that $H = H_1 \cup H_2$ is strictly $2$-balanced, we
have $(e_H-1)/(v_H-2) > (e_{H_1}-1)/(v_{H_1}-2)$.  Therefore, from the
last two inequalities we deduce that
\begin{eqnarray*}
(e_{H_1}-1)(v_H-2) &<& (v_{H_1}-2)(e_H-1),\,\,\,\,\,\,\, \textrm{and} \\
-(e_{H_1}-1)(v_{H_2}-2) &\le& -(v_{H_1}-2)(e_{H_2}-1).
\end{eqnarray*}
Hence, 
\begin{eqnarray*}
(e_{H_1}-1)(v_H-v_{H_2}) < (v_{H_1}-2)(e_H - e_{H_2}).
\end{eqnarray*}
Applying the facts that $v_H = v_{H_1} + v_{H_2} - 2$ and $e_H = e_{H_1}
+ e_{H_2} - 1$ to the last inequality, we get $(e_{H_1}-1)(v_{H_1} - 2) <
(v_{H_1}-2)(e_{H_1} - 1)$, which is a clear contradiction.
\end{proof}

Suppose now that $H$ is regular, strictly $2$-balanced and $v_H \ge 4$.
We use Fact~\ref{fact:f1} in order to generalize the argument for $K_3$
given above.  Define $G_d^+ := G_d \cup \{g_{d-1}\}$. We will show
below that $G_{d-1} \subseteq G_d^+$. Since $H$ is regular, this implies
that $g_{d-2} \in G_d^+$. Since $G_d^+ = G_d \cup \{g_{d-1}\}$ and $g_{d-1}
\ne g_{d-2}$ it then
follows that $g_{d-2} \in G_d$. This, by Definition~\ref{def::afdtree},
contradicts the assumption that $P$ is a directed path in $T_{f,d}$. It
thus remains to show $G_{d-1} \subseteq G_d^+$.
Since every edge in $G_{d-1}$ is adjacent to some vertex $a' \notin \{x,
y\}$, it is enough to show that for every vertex $a' \notin \{x, y\}$ of
$G_{d-1}$, any edge adjacent to $a'$ in $G_{d-1}$ is an edge in $G_d^+$.

Let $a' \notin \{x, y\}$ be an arbitrary vertex of $G_{d-1}$.
By Fact~\ref{fact:f1} there exists in $G_{d-1}$ a path $(a_0=a, a_1,
a_2, \ldots, a_l=a')$ of length $l \ge 0$ from $a$ to $a'$, that avoids
the vertices $\{x, y\}$.
We first claim that $a_i$ is a vertex of $G_d^+$ 
for every $0 \le i \le l$. Indeed, the claim is trivially true for $i=0$, as $a_0 = a$
is a vertex of $G_d \subset G_d^+$.  Assume that for
$0 \le i-1 < l$, $a_{i-1}$ is a vertex of $G_d^+$. 
Note that in that case, $a_{i-1}$ is a vertex of both $G_d$ and $G_{d-1}$ and
that $a_{i-1} \notin \{x, y\}$. Then by the key observation made above, every
edge that is adjacent to $a_{i-1}$ in $G_d$ is an edge in $G_{d-1}$. Also, 
we have $g_{d-1} \in G_{d-1}$.  Therefore, every edge that
is adjacent to $a_{i-1}$ in $G_d^+$ is an edge in $G_{d-1}$. 
This, together with the facts that $H$ is regular and $a_{i-1} \notin \{x,
y\}$ implies that every edge that is adjacent to $a_{i-1}$ in $G_{d-1}$ is 
an edge in $G_d^+$. Therefore, since $\{a_{i-1}, a_i\}$ is an edge in $G_{d-1}$,
we get as needed that $a_i$ is a vertex of $G_d^+$. 

Now that we have found that $a'$ is a vertex of both $G_d^+$ and
$G_{d-1}$, using the fact that $a' \notin \{x, y\}$, we again have by the
key observation above that every edge that is adjacent to $a'$ in $G_d$
is an edge in $G_{d-1}$. Since $g_{d-1} \in G_{d-1}$ we thus have that
every edge that is adjacent to $a'$ in $G_d^+$ is an edge in $G_{d-1}$.
Therefore, by regularity of $H$ and since $a'
\notin \{x, y\}$, we have that every edge that is adjacent to $a'$ in
$G_{d-1}$ is an edge in $G_d^+$.  This completes the proof of the claim.
\end{proof}

Claim~\ref{claim::clm0} gives us property P1 (and part of property~P2). We
now prove that property~P2 holds as well.
\begin{claim}
\label{claim:p2}
For any $d \in [D]$, property~P2 holds for $T_{f,d}$.
\end{claim}
\begin{proof}
The proof is by induction on $d$.  To see that property P2 holds for the
base case, $d=1$, let $(v_0,u_1)$ and $(v_0,u_1')$ be two different paths
in $T_{f,1}$ so that $u_1 \ne u_1'$. Let $G_1$ and $G_2$ be the labels
of the nodes $u_1$ and $u_1'$, and assume for the sake of contradiction
that $G_1 \cap G_2 \ne \emptyset$. Then clearly $G_2$ shares an edge with 
$\{f\} \cup G_1$.  Since $G_1$ and $G_2$ are
distinct subgraphs (or else $u_1 = u_1'$), and using the fact that
$G_2 \in \Lambda(f,\rho)$ and so $G_2 \cap \{f\} = \emptyset$, we
also have that $G_2$ shares at most $e_H-2$ edges with
$\{f\} \cup G_1$. Hence,
by definition, $(G_1,G_2)$ is a bad sequence. This contradicts $E_3$
and so property~P2 holds for $T_{f,1}$.

Let $d \in [D]$, $d \ge 2$ and assume the claim holds
for $d-1$.  We prove that~P2 holds for $T_{f,d}$.  Let $P =
(v_0,u_1,v_1,\ldots,v_{d-1},u_d,v_d)$ be a path in $T_{f,d}$ from
the root $v_0$ to a leaf. For $j \in [d]$, let $G_j$ be the label
of the node $u_j$.  Note that by the induction hypothesis and by
Claim~\ref{claim::clm0}, in order to prove the claim for $T_{f,d}$,
it suffices to show that $G_d$ does not share an edge with the label
of any node $u \notin \{u_1,u_2,\ldots,u_d\}$ in $T_{f,d}$, where $u$
is at odd distance from $v_0$.  Assume for the sake of contradiction
that for some $v_i$, $0 \le i \le d-1$, and some $l \ge 1$, there exists
a path $P' = (v_i,u_1',v_1',\ldots,v_{l-1}',u_l',v_l')$ in $T_{f,d}$
such that $u_1 \ne u_1'$ (and so $u_l' \notin \{u_1,u_2,\ldots,u_d\}$)
and yet the labels of $u_d$ and $u_l'$ have a non-empty intersection. We
will reach a contradiction by constructing a bad sequence from the labels
along the paths $P$ and $P'$.

For $j \in [l]$, let $G_j'$ be the label of node $u_j'$.  We note that
by assumption, $|G_d \cap G_l'| \ge 1$.  We assume for simplicity that
$l$ is minimal, in the sense that $|G_d \cap G_j'| = 0$ for every $1 \le 
j \le l-1$, or else we can shorten the path $P'$ so as to satisfy this
assumption and still construct a bad sequence and get a contradiction.
We consider two cases:
\begin{itemize}
\item Assume $v_l'$ is a node in $T_{f,d-1}$ so that $P'$ is a path in
$T_{f,d-1}$. Define $S$ to be the sequence which is the concatenation
of $(G_1, G_2, \ldots, G_{d-1})$ with $(G_1', G_2', \ldots, G_l')$.
For convenience, rewrite $S = (F_1,F_2,\ldots,F_{l+d-1})$ (here $F_1 = G_1$
and $F_{l+d-1} = G_l'$). From the induction hypothesis and 
claim~\ref{claim::clm0} we have that for every $j \in [l+d-1]$,
$F_j$ shares no edge with $\{f\} \cup \bigcup_{i<j} F_i$.
From this, together with the occurrence of $E_3$ we get that
for all $j \in [l+d-1]$, $F_j$ shares
exactly $2$ vertices with $\{f\} \cup \bigcup_{i<j} F_i$.

We now claim that the sequence $S' = (F_1, F_2, \ldots, F_{l+d-1}, G_d)$
is a bad sequence. First note that by the minimality of $l$, $G_d$ shares
 no edge with $\bigcup_{i < l} G_i'$. Also, by Claim~\ref{claim::clm0}, 
$G_d$ shares no edge with $\{f\} \cup \bigcup_{i < d} G_i$.  In contrast, 
by assumption, $G_d$ shares at least one edge with $G_l'$.  Note that $S'$ 
contains at most $2D$ subgraphs. Hence, to demonstrate that $S'$ is a bad 
sequence and get a contradiction it is enough to show that $G_d$ shares
at most $e_H-2$ edges with $G_l'$.

If $G_d$ shares $e_H-1$ edges with $G_l'$ then $G_d = G_l'$.  In that case,
 since $H$ is regular, we would have that the label of the node $v_{d-1}$ 
is the same as the label of the node $v_{l-1}'$. But since clearly $v_{d-1}
 \ne v_{l-1}'$ (indeed, the distance of $v_{d-1}$ from $v_0$ is strictly 
larger than the distance of $v_{l-1}'$ from $v_0$), this violates either 
property~P2 of $T_{f,d-1}$ which holds by the induction hypothesis, or
 property P1 which we have already established.  Hence $G_d$ shares at most
$e_H-2$ edges with $G_l'$ and we conclude that $S'$ is a bad sequence, a
contradiction to $E_3$.
\item
Assume $v_l'$ is a leaf in $T_{f,d}$.  Define the sequence 
$S = (F_1, F_2, \ldots, F_{l+d-1})$ as in the previous item. From the
induction hypothesis, Claim~\ref{claim::clm0} and from the previous item we 
have that for every $j \in [l+d-1]$, $F_j$ shares no edge with
$\{f\} \cup \bigcup_{i<j} F_i$. By the occurrence of $E_3$ we then have
that for every $j \in [l+d-1]$, $F_j$ shares exactly $2$ vertices with
$\{f\} \cup \bigcup_{i<j} F_i$. 

Define $S' = (F_1, F_2, \ldots, F_{l+d-1}, G_d)$ as in the previous item.
We again claim that $S'$ is a bad sequence. To verify this claim, first
note that by Claim~\ref{claim::clm0}, $G_d$ shares no no edge with $\{f\} 
\cup \bigcup_{i<d} G_i$. By the minimality of $l$, $G_d$ shares no edge 
with $\bigcup_{i<l} G_i'$. Also, there are at most $2D$ subgraphs in $S'$ 
and by assumption, $G_d$ shares some edge with $G_l'$. Hence, to conclude 
that $S'$ is a bad sequence, it is enough to show that $G_d$ shares at most 
$e_H-2$ edges with $G_l'$. We have two cases.
\begin{enumerate}
\item Suppose $v_{d-1} = v_{l-1}'$. Since $u_d \ne u_l'$ we have
by definition that $G_d$ shares at most $e_H-2$ edges with $G_l'$.
\item Suppose $v_{d-1} \ne v_{l-1}'$. If $G_d$ shares all $e_H-1$ 
edges with $G_l'$ then since $H$ is regular, we have that the labels
of $v_{d-1}$ and $v_{l-1}'$ are the same. This in turn implies that 
the parents of $v_{d-1}$ and $v_{l-1}$ are distinct and their labels
share an edge. This is a contradiction to the induction hypothesis. 
Hence, as needed, $G_d$ shares at most $e_H-2$ edges with $G_l'$.
\end{enumerate}
\end{itemize}
We conclude that property~P2 holds for $T_{f,d}$.
\end{proof}

We conclude the proof by arguing that property~P3 holds.  Let
$(v_0,u_1,v_1,\ldots,v_{d-1},u_d,v_d)$ be a path in $T_{f,D}$, starting
from the root $v_0$. Let $G_j$ be the label of node $u_j$ and let $g_j$
be the label of node $v_j$.  We need to show that the number of nodes
adjacent to $v_{d-1}$ in $T_{f,D}$ is $|\Lambda(g_{d-1},\rho)|-O(1)$.
The claim is trivially true for $d=1$, since any subgraph in $\Lambda(f,\rho)$
is a label of a node adjacent to the root of $T_{f,D}$. For $d \ge
2$, we recall that by definition, the nodes that are adjacent to
$v_{d-1}$ in $T_{f,D}$ are those nodes whose labels are in $\{G_d \in
\Lambda(g_{d-1},\rho) : g_{d-2} \notin G_d\}$.  Reflecting on the proof of
Claim~\ref{claim::clm0}, we see that assuming $E_3$, if $g_{d-2} \in G_d
\in \Lambda(g_{d-1},\rho)$, then the set of vertices of $G_d$ is the same as
that of $G_{d-1}$. This immediately implies that $|\{G_d \in \Lambda(g_{d-1},\rho) : g_{d-2} \in G_d\}|
= O(1)$.  This gives us property P3.  With that, we conclude the proof.
\end{proof}

\begin{proposition}
\label{eq:pe3}
$\prob{E_3} = 1-o(1)$.
\end{proposition}
\begin{proof}
Let $Z$ be the random variable counting the number of bad sequences $S$
for which $\event{S \subseteq \GG(n,\rho)}$ occurs.  Since the probability
that $\event{S \subseteq \GG(n,\rho)}$ occurs for some bad sequence $S$
is at most $\expec{Z}$, showing that $\expec{Z} = o(1)$ would imply
the proposition.

\def\Seq{\textrm{Seq}}
For $d \ge 2$, let $\Seq_d$ denote the collection of all bad sequences
of length $d$. Then
\begin{eqnarray}
\label{eq::sum2}
\expec{Z} &=& \sum_{2 \le d \le 2D} \sum_{S_d \in \Seq_d} \Prob{S_d \subseteq \GG(n,\rho)}.
\end{eqnarray}
Below we show that for every $d$ satisfying $2 \le d \le 2D$,
\begin{eqnarray}
\label{eq::pil}
\sum_{S_d \in \Seq_d} \Prob{S_d \subseteq \GG(n,\rho)} \le n^{-\epsilon_H+o(1)},
\end{eqnarray}
where $\epsilon_H>0$ is the constant provided in Fact~\ref{fact::ufact}.
From~(\ref{eq::sum2})~and~(\ref{eq::pil}) and since $2D = n^{o(1)}$, we get
that $\expec{Z} \le n^{-\epsilon_H+o(1)} = o(1)$ as required. Hence it remains to
prove~(\ref{eq::pil}).

Fix for the rest of the proof $d$ satisfying $2 \le d \le 2D$.  In order
to prove~(\ref{eq::pil}), it would be convenient to partition the set
$\Seq_d$ to two parts and then upper bound the sum in~(\ref{eq::pil}) for
each of the two parts of the partition. Let $\Seq_{d,1}$ be the set of
all sequences $S = (G_i)_{i=1}^d \in \Seq_d$ for which $G_d$ shares
no edge (and at least $3$ vertices) 
with $\{f\} \cup \bigcup_{i<d} G_i$. Let $\Seq_{d,2}$ be the set of
all sequences $S = (G_i)_{i=1}^d \in \Seq_d$ for which $G_d$ shares at least
$1$ edge (and at most $e_H-2$ edges) with $\{f\} \cup \bigcup_{i<d} G_i$.
%
%
It would be useful to further classify those members of $\Seq_{d,2}$
as follows.  Suppose that $S = (G_i)_{i=1}^d \in \Seq_{d,2}$.
Let $g$ be the unique edge in $\{f\} \cup \bigcup_{i < d} G_i$
such that $G_d \in \Lambda(g,1)$. 
(The edge $g$ is unique, since $H$ is regular.)  Define
\begin{displaymath}
H_{S} := \{g\} \cup \Big(
                    G_d \cap \big(
                    \{f\} \cup \bigcup_{i<d} G_i
                             \big)
                    \Big).
\end{displaymath}
Then we say that $S$ is an
\emph{$F$-type} if $H_{S}$ is isomorphic to $F$.  Note that if $S$
is an $F$-type then $F$ is a proper subgraph of $H$ with $v_F \ge 3$.
We are now ready to upper bound the sum in~(\ref{eq::pil}) for the two
parts of the partition of $\Seq_d$.

We start by giving an upper bound on the sum in~(\ref{eq::pil}) when the
sum ranges over all $S \in \Seq_{d,1}$.  First, we upper bound the size of 
$\Seq_{d,1}$.  To do so, we construct a bad sequence
$S = (G_i)_{i=1}^d \in \Seq_{d,1}$ iteratively.  Assume we have already chosen
the first $j-1$ subgraphs in $S$ for $j \in [d-1]$.
We count the number of choices for $G_j$: There are $O(d)$
possible choices for an edge $g \in \{f\} \cup \bigcup_{i<j} G_i$ for which
$G_j$ is in $\Lambda(g,1)$; There are at most $n^{v_H-2}$ choices
for the vertices of $G_j$; There are $O(1)$ choices for the edges of
$G_j$ given that we have already fixed its $v_H$ vertices. In total,
the number of choices for $G_j$ is at most $O(d \cdot n^{v_H-2}) =
n^{v_H-2 + O((\ln\ln n)/\ln n)}$.
Assume we have already chosen the first $d-1$
subgraphs in $S$. We count the number of choices for $G_d$: There
are $O(d)$ possible choices for an edge $g \in \{f\} \cup \bigcup_{i<d} G_i$
for which $G_d$ is in $\Lambda(g,1)$; There are at most $O(n^{v_H-3}
\cdot d^{v_H})$ choices for the vertices of $G_d$, which follows from
the fact that $G_d$ shares at least $3$ vertices with $\{f\} \cup \bigcup_{i<d}
G_i$; There are $O(1)$ choices for the edges of $G_d$, given that we have fixed
the $v_H$ vertices of $G_d$. In total, the number of choices for $G_d$
is at most $O(d \cdot n^{v_H-3} \cdot d^{v_H}) = n^{v_H-3 + o(1)}$.
From the above we conclude that the number of bad sequences $S \in \Seq_{d,1}$
is at most $n^{d(v_H-2)-1+o(1)}$.
Now, it is easy to see that if $S \in \Seq_{d,1}$
then the probability of $\event{S \subseteq \GG(n,\rho)}$ is
$\rho^{d(e_H-1)}$.  Thus, recalling Definition~\ref{def::kk}, we conclude that
\begin{eqnarray}
\nonumber
\sum_{S \in \Seq_{d,1}} \Prob{S \subseteq \GG(n,\rho)}
&\le& n^{d(v_H-2)-1+o(1)}\rho^{d(e_H-1)} \\
\nonumber
&\le& k^{d(e_H-1)}n^{-1+o(1)} \\
\label{eq::poi1}
& = & n^{-1+o(1)}.
\end{eqnarray}

We now upper bound the sum in~(\ref{eq::pil}) when the sum ranges over all
$S \in \Seq_{d,2}$. We first upper bound the number of $F$-type bad sequences
$S \in \Seq_{d,2}$, for some $F \subsetneq H$, $v_F \ge 3$.  As before, we 
construct such a bad sequence $S = (G_i)_{i=1}^{d}$ iteratively.  Assume we
have already chosen the first $j-1$ subgraphs in $S$, for $j \in [d-1]$. Then
the number of choices for $G_j$ is, as before, $n^{v_H-2+O((\ln\ln n)/\ln n)}$.  Assume 
we have already chosen the first $d-1$ subgraphs in $S$. We 
count the number of choices for $G_d$: There are $O(d)$ choices for an
edge $g \in \{f\} \cup \bigcup_{i<d} G_i$ such that $G_d \in \Lambda(g,1)$;
There are at most $O(n^{v_H-v_F} \cdot d^{v_F})$ choices for the vertices
of $G_d$; There are $O(1)$ choices for the edges of $G_d$, given that
we have fixed its edges. In total, the number of choices for $G_d$
is $O(d \cdot n^{v_H-v_F} \cdot d^{v_F}) = n^{v_H-v_F + o(1)}$.
Hence, we conclude that the number of $F$-type bad sequences $S \in
\Seq_{d,2}$ is $n^{(d-1)(v_H-2) + v_H-v_F+o(1)}$.  For such a sequence
$S$, it can be verified that the probability of $\event{S \subseteq
\GG(n,\rho)}$ is $\rho^{(d-1)(e_H-1)+e_H-e_F}$.  Hence, taking
$\sum_F$ to be the sum over all $F \subsetneq H$ with $v_F
\ge 3$ and using Fact~\ref{fact::ufact}, we have 
\begin{eqnarray}
\nonumber
\sum_{S \in \Seq_{d,2}} \Prob{S \subseteq \GG(n,\rho)}
&\le& \sum_{F} n^{(d-1)(v_H-2)+v_H-v_F+o(1)}\rho^{(d-1)(e_H-1)+e_H-e_F} \\
\nonumber
&\le& \sum_F k^{(d-1)(e_H-1)}n^{-\epsilon_H+o(1)}  \\
\label{eq::poi2}
&=& \sum_F n^{-\epsilon_H+o(1)}.
\end{eqnarray}
Since $\sum_F 1 = O(1)$, we conclude from~(\ref{eq::poi1}) 
and~(\ref{eq::poi2}) the validity of~(\ref{eq::pil}).  This completes 
the proof.
\end{proof}

\section{Proof of Lemma~\ref{lemma::toprove2}}
\label{sec::4}
In order to prove Lemma~\ref{lemma::toprove2}, let us fix an edge $f
\in K_n$ and assume everywhere throughout the section that $E_2(f)
\cap E_1$ occurs.  Hence, we may fix once and for the rest of this
section the good tree $T=T_{f,D}$ which is guaranteed to exist by the
occurrence of $E_2(f)$. It now suffices to lower bound the probability
of $\eventcnd{f \in \MM_n(H)}{\beta(f) < cn^{-(v_H-2)/(e_H-1)}}$.
It is extremely important to note, and we use this fact implicitly
throughout this section, that for every edge $g$ that appears as a label
of a non-root node at even height in $T$, the probability of the event
$\event{\beta(g) < \rho'}$, for $\rho' \le \rho$, is $\rho'/\rho$. This
is true since for such an edge $g$ we already condition on the event
$\event{g \in \GG(n,\rho)}$ and so $\beta(g)$ is uniformly distributed
in $[0,\rho]$. Therefore, for example, if $g$ is a label of a non-root
node at even height in $T$, then the probability of $\event{\beta(g) <
cn^{-(v_H-2)/(e_H-1)}}$ is $c/k$.

A useful convention we use throughout this section is this: Every two
distinct nodes in the good tree $T$ have distinct labels and so, for
the rest of this section we will refer to the nodes of $T$ \emph{by
their labels}.  We also recall and introduce some useful notation.
First recall that for every node $u$ in $T$, $\Gamma_{T}(u)$ denotes the
set of nodes adjacent to $u$ in $T$. For simplicity, we shall replace
$\Gamma_T(u)$ with $\Gamma(u)$ when no confusion arises.  Recall also
that the tree $T$ has height $2D$.  If $g$ is a node at height $2d$ in
$T$, we denote by $T_{g,d}$ the subtree of $T$ that is rooted at $g$.
Lastly, if $B \subseteq [0,1]^n$ is any event, we denote by $\ol{B}$
the event $[0,1]^n \setminus B$.

The overall structure of the proof of Lemma~\ref{lemma::toprove2} is this:
First, in the next few paragraphs, conditioning on $\event{\beta(f)
< cn^{-(v_H-2)/(e_H-1)}}$, we reduce the problem of lower bounding
the probability of $\event{f \in \MM_n(H)}$ to the problem of lower
bounding the probability of a certain event, $\ol{B(T_{f,D})}$ (to be
defined shortly), which we define using the tree $T_{f,D}$. We are then
left with the task of estimating the probability of $\ol{B(T_{f,D})}$,
a task being performed in Sections~\ref{sec::4a}~and~\ref{sec::4b}.

\begin{definition}
\label{def::block}
Let $g$ be a node in $T$ at height $2d$ and let $T_{g,d}$ be the subtree
of $T$ rooted at $g$.  Define
\begin{eqnarray*}
B(T_{g,d}) := \left\{ \begin{array}{ll}
\emptyset & \textrm{if $d = 0$,} \\
\exists G' \in \Gamma(g).\,\,\,
 \forall g' \in G'.\,\,\,  
 \event{\beta(g')<\beta(g)} \cap \ol{B(T_{g',d-1})} 
 & 
         \textrm{if $1 \le d \le D$.}
\end{array} \right.
\end{eqnarray*}
\end{definition}

\begin{proposition}
\label{prop::pqp}
Let $g$ be a non-leaf node in $T$ at height $2d$ and let $T_{g,d}$ be
the subtree of $T$ rooted at $g$.  If $d$ is even then $B(T_{g,d})$
implies $\event{g\notin \MM_n(H)}$.
\end{proposition}
\begin{proof}
The proof is by induction on $d$.  Since $g$ is a non-leaf and $d$
is even, we start with the case $d=2$ (so the distance from $g$ to
any leaf of $T_{g,d}$ is $4$).  Assume $B(T_{g,2})$ occurs. Then by
Definition~\ref{def::block} there exists  a graph $G' \in \Gamma(g)$
for which the following two events occur: (i)  for every $g' \in
G'$, $\event{\beta(g')<\beta(g)}$, and (ii) for every $g' \in G'$,
$\ol{B(T_{g',1})}$ occurs.
Now observe that in order to conclude that $\event{g\notin \MM_n(H)}$
occurs, it suffices to show that $\event{G' \subseteq \MM_n(H)}$ occurs.
To show the occurrence of the last event, it suffices to show that for
every edge $g' \in G'$ and for every graph $G'' \in \Lambda(g',\rho)$,
there exists an edge $g'' \in G''$ whose birthtime $\beta(g'')$ is larger
than $\beta(g')$.
Let us fix $g' \in G'$ and $G'' \in \Lambda(g',\rho)$.  We have two
cases.  If $G'' \in \Gamma(g')$ then by the fact that $\ol{B(T_{g',1})}$
occurs, we have indeed that there exists an edge $g'' \in G''$ such that
$\beta(g'') > \beta(g')$ (here we have used the fact that with probability
$1$, $\beta(g'') \ne \beta(g')$).  If on the other hand $G'' \notin
\Gamma(g')$ then by definition of $T$, $g \in G''$. Then, by item~(i)
above, we have that for some $g'' \in G''$, $\beta(g'') = \beta(g)
> \beta(g')$.  We thus conclude that $\event{g\notin \MM_n(H)}$ occurs.

Assume the proposition is valid for $d-2$. We prove it for $d$, so assume
that $B(T_{g,d})$ occurs. This implies by Definition~\ref{def::block}
that there exists  a graph $G' \in \Gamma(g)$ for which the following two
events occur: (i)  for every $g' \in G'$, $\event{\beta(g')<\beta(g)}$,
and (ii) for every $g' \in G'$, $\ol{B(T_{g',d-1})}$ occurs.
As before, to prove the proposition it suffices to show that $\event{G'
\subseteq \MM_n(H)}$ occurs. The occurrence of this last event can be
proved if we show that for every edge $g' \in G'$ and for every graph
$G'' \in \Lambda(g',\rho)$, there exists an edge $g'' \in G''$ such that
\emph{either} $\beta(g'') > \beta(g')$ \emph{or} $g'' \notin \MM_n(H)$.
Fix $g' \in G'$ and $G'' \in \Lambda(g',\rho)$.  As before, we have
two cases.  Suppose first that $G'' \in \Gamma(g')$. Then there exists
an edge $g'' \in G''$ such that either $\beta(g'') > \beta(g')$ or
$B(T_{g'',d-2})$ occurs.  This implies by the induction hypothesis that
either $\beta(g'') > \beta(g')$ or  $g'' \notin \MM_n(H)$, as required.
The second case is that of $G'' \notin \Gamma(g')$. Similarly to the base
case, this implies that there exists an edge $g'' \in G''$ (specifically
$g'' = g$) such that $\beta(g'') > \beta(g')$. Hence $\event{g \notin
\MM_n(H)}$ occurs.
\end{proof}

Recall Definition~\ref{def::kk} and note that $D$ is odd.  Hence, from
Definition~\ref{def::block} and Proposition~\ref{prop::pqp}, it follows
that, conditioning on $\event{\beta(f) < cn^{-(v_H-2)/(e_H-1)}}$,
$\ol{B({T_{f,D}})}$ implies $\event{f \in \MM_n(H)}$. (Indeed,
note that $\Gamma(f) = \Lambda(f,\rho)$ by definition.) Hence, we
get that the probability of $\eventcnd{f \in \MM_n(H)}{\beta(f) <
cn^{-(v_H-2)/(e_H-1)}}$ is lower bounded by
\begin{eqnarray}
\label{eq::todo}
\Probcnd{\ol{B(T_{f,D})}}{\beta(f) < cn^{-(v_H-2)/(e_H-1)}}.
\end{eqnarray}
In the two subsections below we lower bound~(\ref{eq::todo}). Specifically, we
show that
\begin{eqnarray}
\label{eq::todox}
\Probcnd{\ol{B(T_{f,D})}}{\beta(f) < cn^{-(v_H-2)/(e_H-1)}}  =
  \Omega\bigg(\frac{(\ln c)^{1/(e_H-1)}}{c}\bigg),
\end{eqnarray}
which, given the discussion above, proves Lemma~\ref{lemma::toprove2}.
In order to prove~(\ref{eq::todox}) it would be convenient to first
restrict ourselves to the following special case.  Consider the tree $RT_{f,d}$
which is obtained from $T_{f,d}$ as follows: For every non-leaf node
$g$ at even height in $T_{f,d}$, remove an arbitrary subset of the
subtrees rooted at the nodes adjacent to $g$, so that the outdegree
of $g$ becomes $\floor{\lambda k^{e_H-1}} - \floor{k^{e_H/2-1/3}}$
\emph{exactly}. Note that this can be done, as we assume that $E_1$
occurs. For a node $g$ at height $2d$ in $RT:=RT_{f,D}$, we denote by
$RT_{g,d}$ the subtree of $RT$ that is rooted at $g$.
Given this definition the task of lower
bounding~(\ref{eq::todo}) is now divided to two parts.  In the first part
(Section~\ref{sec::4a}), we prove~(\ref{eq::todox}) for the special
case that $T_{f,D}$ is replaced with $RT_{f,D}$.  In the second part
(Section~\ref{sec::4b}), we show that asymptotically, the probability
of $\eventcnd{\ol{B(T_{f,D})}}{\beta(f) < cn^{-(v_H-2)/(e_H-1)}}$
is equal to the probability of $\eventcnd{\ol{B(RT_{f,D})}}{\beta(f)
< cn^{-(v_H-2)/(e_H-1)}}$.  Combining these two parts, we will
conclude the validity of~(\ref{eq::todox}) and so also the validity of
Lemma~\ref{lemma::toprove2}.

\subsection{Analyzing $\ol{B(RT)}$}
\label{sec::4a}
For $x<y$, let $p_d(x,y)$ be the probability of $\ol{B(RT_{f,d})}$ 
conditioned on $\event{xn^{-(v_H-2)/(e_H-1)} \le
\beta(f) < yn^{-(v_H-2)/(e_H-1)}}$. The main result in this subsection follows.
\begin{lemma}
\label{lemma:brt}
$p_D(0,c) = \Omega\big(\frac{(\ln c)^{1/(e_H-1)}}{c}\big)$.
\end{lemma}
To prove Lemma~\ref{lemma:brt}, we begin with the following proposition.
\begin{proposition}
\label{lemma::translate}
For $0 < x \le c$ and $n$ sufficiently large, 
$p_{D}(0,x) \ge p_{D-1}(0,x) - 2^{-D+1}$.
\end{proposition}
\begin{proof}
Define
\begin{eqnarray*}
\G_i := \left\{ \begin{array}{ll}
 \Big\{ G_1 \in \Gamma_{RT}(f) : \forall g_1 \in G_1.\,\,\,
          \beta(g_1) < cn^{-(v_H-2)/(e_H-1)}\Big\}
 & \textrm{if $i = 1$,} \\
\bigcup_{g_{i-1} \in G_{i-1}\in \G_{i-1}}
  \Big\{ G_i \in \Gamma_{RT}(g_{i-1}) : \forall g_i \in G_i.\,\,\,
          \beta(g_i) < \beta(g_{i-1})\Big\}
 & 
         \textrm{if $2 \le i \le d$.}
\end{array} \right.
\end{eqnarray*}
By inspecting Definition~\ref{def::block}, one sees that conditioned on
$\event{\beta(f) < xn^{-(v_H-2)/(e_H-1)}}$, $\ol{B(RT_{f,D-1})}$
implies $\ol{B(RT_{f,D})} \cup \{\G_{D-1} \ne \emptyset\}$. Hence $p_{D-1}(0,x)
\le p_{D}(0,x) + \prob{\G_{D-1} \ne \emptyset}$ and so it suffices to prove that
$\prob{\G_{D-1} \ne \emptyset} \le 2^{-D+1}$. For that, we show that
the expected size of $\G_{D-1}$ is at most $2^{-D+1}$.

To estimate the expected size of $\G_{D-1}$, we first give an upper bound on 
the probability that a given subgraph $G_{D-1}$ is included in the set
 $\G_{D-1}$. For that, let us consider the unique path 
$(f=g_0,G_1,g_1,\ldots,g_{D-2},G_{D-1})$ from the root of $RT_{f,D}$ to
 $G_{D-1}$. Then $G_1$ is included in $\G_1$ if and only if 
$\beta(g) < cn^{-(v_H-2)/(e_H-1)}$ for every $g \in
G_1$. For $j \ge 2$, $G_j$ is included in $\G_j$ if and only if $G_i$
is included in $\G_i$ for every $i < j$ and in addition, $\beta(g) <
\beta(g_{j-1})$ for every $g \in G_j$.  Now, the probability of $\event{G_1
\in \G_1}$ is exactly $(c/k)^{e_H-1}$. Given $G_1 \in \G_1$, we have that
$g_1$ is uniformly distributed in $[0,cn^{-(v_H-2)/(e_H-1)})$.  Moreover,
for $j \ge 2$, given that $G_{j} \in \G_{j}$, we
have that $g_j$ is uniformly distributed in $[0,\beta(g_{j-1}))$. Hence,
it follows that 
\begin{eqnarray}
\nonumber
\prob{G_{D-1} \in \G_{D-1}} &=&
\int_0^1 \ldots \int_0^{1}
\Big(\frac{c}{k}\Big)^{e_H-1} \cdot 
\prod_{j=2}^{{D-1}}
\Big(\frac{c}{k} \cdot \prod_{i=1}^{j-1}x_i\Big)^{e_H-1} \,\,\,
{d}x_1 \,
{d}x_2 \, \ldots \,
{d}x_{D-2} \\
\nonumber
&=& \frac{c^{(e_H-1)(D-1)}}{k^{(e_H-1)(D-1)}}
\cdot \frac{1}{ 
          \prod_{i=1}^{D-2} (i(e_H-1)+1)
               } \\
\label{eq::integ}
&<&
\frac{c^{(e_H-1)(D-1)}}{k^{(e_H-1)(D-1)} (D-1)!},
\end{eqnarray}
where the inequality follows from $e_H \ge 3$.  Since there are no
more than $(e_H \lambda k^{e_H-1})^{D-1}$ nodes at distance $2(D-1)-1$ from
the root of $RT$, we deduce from~(\ref{eq::integ}) that the expected
number of nodes in $\G_{D-1}$ is at most $(e_H \lambda 
c^{e_H-1})^{D-1}/((D-1)!)$.  Since $(D-1)! \ge ((D-1)/3)^{D-1}$ and 
since $D-1 \ge 6 e_H \lambda c^{e_H-1}$ for $n$ sufficiently large,
the expected number of nodes in $\G_{D-1}$ is at most
\begin{eqnarray*}
\frac{(3 e_H \lambda c^{e_H-1})^{D-1}}{(D-1)^{D-1}} \le 
\frac{(3 e_H \lambda c^{e_H-1})^{D-1}}{(6 e_H \lambda c^{e_H-1})^{D-1}} =
2^{-D+1},
\end{eqnarray*}
as required.  This completes the proof.
\end{proof}
We also need the following fact.
\begin{proposition}
\label{fact::mono}
Let $x<y$. Then
$p_d(0,x) \ge p_d(0,y)$ and $p_d(0,x) \ge p_d(x,y)$.
\end{proposition}
\begin{proof}
Since $p_d(0,y)$ is the weighted average of $p_d(0,x)$ and $p_d(x,y)$, it
is enough to show that $p_d(0,x) \ge p_d(x,y)$.
Take any birthtime function $\beta$ under which $\ol{B(RT_{f,d})}$
occurs, with $xn^{-(v_H-2)/(e_H-1)} \le \beta(f) < yn^{-(v_H-2)/(e_H-1)}$.
Now alter $\beta(f)$ so that we have $\beta(f) < xn^{-(v_H-2)/(e_H-1)}$. It
is easy to check given Definition~\ref{def::block} that $\ol{B(RT_{f,d})}$
occurs even after the above alteration of $\beta$. This implies that
$p_d(0,x) \ge p_d(x,y)$.
\end{proof}

We now turn to lower bound $p_D(0,c)$.  Let $g \in G \in \Gamma_{RT}(f)$.  Condition on
the occurrence of $\event{in^{-(v_H-2)/(e_H-1)}\le \beta(f) < (i+1)n^{-(v_H-2)/(e_H-1)}}$ 
for some $i \in [c-1]$. Under that condition we have $\prob{\beta(g) < \beta(f)} \le (i+1)/k$. 
If we further condition on $\event{\beta(g) < \beta(f)}$ then we have
by Proposition~\ref{fact::mono} (and the fact that $RT_{g,D-1}$ is isomorphic to
$T_{f,D-1}$) that the probability of $\ol{B(RT_{g,D-1})}$ is upper bounded by
$p_{D-1}(0,i)$. From these observations, using Definition~\ref{def::block} and
the fact that $RT_{f,D}$ is a subtree of the good tree $T_{f,D}$, we get that
for all $i \in [c-1]$,
\begin{eqnarray*}
p_D(i,i+1) \ge 
  \bigg(
     1 - \bigg(\frac{(i+1) \cdot p_{D-1}(0,i)}{k}\bigg)^{e_H-1}
  \bigg)^{\floor{\lambda k^{e_H-1}} - \floor{k^{e_H/2-1/3}}}.
\end{eqnarray*}
Hence, by Proposition~\ref{lemma::translate} we get
\begin{eqnarray}
\nonumber
p_D(0,c) &\ge& \frac{1}{c} \, \sum_{i=\ceil{c/2}}^{c-1} p_D(i,i+1) \\
\nonumber
       &\ge& \frac{1}{c} \, \sum_{i=\ceil{c/2}}^{c-1} 
                   \bigg( 1 - \bigg(\frac{(i+1) \cdot 
         p_{D-1}(0,i)}{k}\bigg)^{e_H-1} \bigg)^{\floor{\lambda k^{e_H-1}} - \floor{k^{e_H/2-1/3}}} \\
\label{eq:sum}
       &\ge& \frac{1}{c} \, \sum_{i=\ceil{c/2}}^{c-1} 
                   \bigg( 1 - \bigg(\frac{(i+1) \cdot 
         (p_D(0,i)+2^{-D+1})}{k}\bigg)^{e_H-1} \bigg)^{\floor{\lambda k^{e_H-1}} - \floor{k^{e_H/2-1/3}}}.
\end{eqnarray}
Define 
\begin{eqnarray*}
\tau(i) = \frac{((100\lambda)^{-1}\ln i)^{1/(e_H-1)}}{i+1} - 2^{-D+1}.
\end{eqnarray*}
We have two cases. First assume that $p_D(0,i) \ge \tau(i)$
for some integer $i$, $\ceil{c/2} \le i \le c-1$.
In that case the proof is complete, since by Proposition~\ref{fact::mono}, 
we get 
\begin{eqnarray*}
p_D(0,c) \ge \frac{p_D(0,\ceil{c/2})}{2} \ge \frac{p_D(0,i)}{2} \ge \frac{\tau(i)}{2} = 
\Omega\Big(\frac{(\ln c)^{1/(e_H-1)}}{c}\Big).
\end{eqnarray*}
Next, assume that $p_D(0,i) < \tau(i)$ for all integers $i$, $\ceil{c/2} \le i \le c-1$.
In that case, by replacing $p_D(0,i)$ with $\tau(i)$ in the sum above, we get
\begin{eqnarray*}
p_D(0,c) \ge 
 \frac{1}{c} \, \sum_{i=\ceil{c/2}}^{c-1} 
  \bigg( 1 - \frac{\ln i}{100\lambda k^{e_H-1}}\bigg)^{\lambda k^{e_H-1}} = 
\Omega\Big(\frac{(\ln c)^{1/(e_H-1)}}{c}\Big).
\end{eqnarray*}

Note that Lemma~\ref{lemma:brt} gives the validity of~(\ref{eq::todox}) for
 the case where $T_{f,D}$ is replaced with $RT_{f,D}$. 
In the next subsection we show that asymptotically, the
probability of $\ol{B(T_{f,D})}$ conditioned on $\event{\beta(f) <
cn^{-(v_H-2)/(e_H-1)}}$ is equal to $p_D(0,c)$.  This will prove~(\ref{eq::todox}) 
and hence also Lemma~\ref{lemma::toprove2}.

\subsection{On $\ol{B(RT)}$ versus $\ol{B(T)}$}
\label{sec::4b}
For $x<y$ and for a node $g$ at height $2d$ in $T$, let
$q_{g,d}(x,y)$ be the probability of $\ol{B(T_{g,d})}$
conditioned on $\event{xn^{-(v_H-2)/(e_H-1)} \le \beta(g) < yn^{-(v_H-2)/(e_H-1)}}$.
The main result in this subsection follows.
\begin{lemma}
\label{lemma:asmeq}
$(1-o(1)) \cdot p_D(0,c) \le q_{f,D}(0,c) \le (1+o(1)) \cdot p_D(0,c)$.
\end{lemma}
Note that Lemma~\ref{lemma:asmeq} together with Lemma~\ref{lemma:brt}
implies~(\ref{eq::todox}), which in turn implies Lemma~\ref{lemma::toprove2}.
We begin the proof of Lemma~\ref{lemma:asmeq} by giving a few useful
definitions which we use throughout the proof.
\begin{definition}
\mbox{}
\begin{itemize}
\item Define for every integer $d \le D$, 
\begin{displaymath}
\epsilon_d := \frac{(128e_Hc)^{2de_H}}{k^{e_H/2-2/3}}.  
\end{displaymath}
\item Fix $\delta > 0$ which satisfies the following:
(i)~$\delta \le k^{-20e_H}$, (ii)~$\delta^{1/4} = o(\epsilon_d \cdot p_{d}(0,c))$ for all $d \in [D]$, 
(iii)~$(1-(2\delta^{1/4}/k)^{e_H-1})^{2k^{e_H-1}} \ge 1-\epsilon_d/2$ for all $d \in [D]$, and (iv)~$1/\sqrt{\delta} \in \N$.
\item Define $J := \{\sqrt{\delta}, \sqrt{\delta}+\delta, \sqrt{\delta}+2\delta, \ldots, c-2\delta, c-\delta\}$.\
\end{itemize}
\end{definition}
\begin{proposition}
\label{prop::pa1}
Let $g$ be a node at height $2d$ in $T$, $0 \le d \le D$.
Assume that for all $j \in J$,
\begin{eqnarray*}
(1-\epsilon_d/2) \cdot p_d(j,j+\delta) \le q_{g,d}(j,j+\delta) \le (1+\epsilon_d/2) \cdot p_d(j,j+\delta).
\end{eqnarray*}
Then for all $j \in J \cup \{c\}$,
\begin{eqnarray*}
(1-\epsilon_d) \cdot p_d(0,j) \le q_{g,d}(0,j) \le (1+\epsilon_d) \cdot p_d(0,j).
\end{eqnarray*}
\end{proposition}
\begin{proof}
First note that the conclusion in the proposition is trivially true for $d=0$, since
$q_{g,d}(0,j) = p_d(0,j) = 1$ for all $j \in J \cup \{c\}$. Hence, we may assume that $d \ge 1$.

Fix $j \in J \cup \{c\}$. Assume first that $j > \sqrt{\delta} + \delta^{1/4}$.
Trivially we have:
\begin{eqnarray*}
p_{d}(0,j) &=& O\Big(\frac{\sqrt{\delta}}{j}\Big) +
                \frac{\delta}{j} \cdot \sum_{j' \in J : j' < j} p_{d}(j',j'+\delta), \\
q_{g,d}(0,j) &=& O\Big(\frac{\sqrt{\delta}}{j}\Big) + \frac{\delta}{j} \cdot \sum_{j' \in J : j' < j} q_{g,d}(j',j'+\delta).
\end{eqnarray*}
Since $j > \sqrt{\delta} + \delta^{1/4}$ we have $\sqrt{\delta}/j \le \delta^{1/4}$. By definition,
$\delta^{1/4} = o(\epsilon_{d} \cdot p_{d}(0,c))$. By Proposition~\ref{fact::mono} and
since $j \le c$, we have $p_{d}(0,c) \le p_{d}(0,j)$.  Therefore, $\sqrt{\delta}/j \le \delta^{1/4} =
o(\epsilon_d \cdot p_{d}(0,j))$.  This, it can be verified, 
together with the two equalities above and with the assumptions given in the
proposition, implies the validity of the claim for $j$.

Next, assume $j \le \sqrt{\delta} + \delta^{1/4}$. Crudely, $q_{g,d}(0,j)$ is lower bounded by the
probability of the event that for every $G \in \Gamma_T(g)$ there is an edge $g' \in G$ with
$\event{\beta(g') > \beta(g)}$. Hence, since $|\Gamma_T(g)| \le 2k^{e_H-1}$ and by definition of $\delta$,
\begin{eqnarray*}
q_{g,d}(0,j) \ge \Big(1 - \Big(\frac{2\delta^{1/4}}{k}\Big)^{e_H-1}\Big)^{2k^{e_H-1}} \ge 1-\epsilon_d/2.
\end{eqnarray*}
Trivially, $p_{d}(0,j) \le 1$. Hence, it follows that $q_{g,d}(0,j)$ is lower bounded
by $(1-\epsilon_d) \cdot p_{d}(0,j)$, as required.
The argument for the upper bound is similar. Trivially, $q_{g,d}(0,j) \le 1$.
Also, by an argument similar to the one used above, $p_d(0,j) \ge 1-\epsilon_d/2$. 
Hence, one can verify that indeed $q_{g,d}(0,j) \le (1+\epsilon_d) \cdot p_d(0,j)$. Thus, 
the claim is valid for all $j \in J \cup \{c\}$.
\end{proof}
The following proposition, when combined with Proposition~\ref{prop::pa1}, implies Lemma~\ref{lemma:asmeq}.
\begin{proposition}
\label{prop:final}
Let $g$ be a node at height $2d$ in $T$, $0 \le d \le D$. Then for all $j \in J$,
\begin{eqnarray*}
(1-\epsilon_d/2) \cdot p_d(j,j+\delta) \le q_{g,d}(j,j+\delta) \le (1+\epsilon_d/2) \cdot p_d(j,j+\delta).
\end{eqnarray*}
\end{proposition}

The rest of the paper is dedicated for the proof of Proposition~\ref{prop:final}.
We collect some useful facts:
\begin{claim}
\label{prop:facts}
Let $g$ be a node at height $2d$ in $T$, $1 \le d \le D$. Then for all $j \in J$,
\begin{enumerate}
\item[(1)] $q_{g,d}(j,j+\delta) \le \prod_{G \in \Gamma_T(g)} \big(1 - \prod_{g' \in G}\frac{j \cdot q_{g',d-1}(0,j)}{k}\big)$.
\item[(2)] $q_{g,d}(j,j+\delta) \ge (1-\epsilon_{d-1}) \cdot
  \prod_{G \in \Gamma_T(g)} \big(1 - \prod_{g' \in G}\frac{j \cdot q_{g',d-1}(0,j)}{k}\big)$.
\item[(3)] $p_d(j,j+\delta) \,\,\, \ge  (1+\epsilon_{d-1})^{-1} \cdot
  \big(1 - \big( \frac{j \cdot p_{d-1}(0,j)}{k}\big)^{e_H-1} \big)^{|\Gamma_{RT}(f)|}$.
\item[(4)] $p_d(j,j+\delta) \,\,\, \le 
  \big(1 - \big(\frac{j \cdot p_{d-1}(0,j)}{k}\big)^{e_H-1})^{|\Gamma_{RT}(f)|}$.
\end{enumerate}
\end{claim}
\begin{proof}
\begin{enumerate}
\item[(1)] 
Condition on $\event{jn^{-(v_H-2)/(e_H-1)} \le \beta(g) < (j+\delta)n^{-(v_H-2)/(e_H-1)}}$ and fix
$g' \in G \in \Gamma_T(g)$.
Given the definition of $\ol{B(T_{g,d})}$ and the fact
that $T$ is a good tree, it is enough to verify that the probability of
the event $\event{\beta(g') < \beta(g)} \cap \ol{B(T_{g',d-1})}$ is at least 
$j \cdot q_{g',d-1}(0,j)/k$.
Indeed, the event in question is implied by the occurrence of $\event{\beta(g') < jn^{-(v_H-2)/(e_H-1)}}
\cap \ol{B(T_{g',d-1})}$.
\item[(2)] Let $\E_1$ be the event that for all $G \in \Gamma_{T}(g)$, it
does not hold that for all $g' \in G$, 
both $\event{\beta(g') < jn^{-(v_H-2)/(e_H-1)}}$ and
$\ol{B(RT_{g',d-1})}$ occur. Let $\E_2$ be the event that for all
$g' \in G \in \Gamma_{T}(g)$, $\beta(g')$ is not in the interval
$[jn^{-(v_H-2)/(e_H-1)}, (j+\delta)n^{-(v_H-2)/(e_H-1)})$. 
It is easy to verify that $q_{g,d}(j,j+\delta) \ge \prob{\E_1 \cap \E_2}$. 
Now, we have $\prob{\E_1} = \prod_{G \in \Gamma_T(g)}
\big(1 - \prod_{g' \in G} \frac{j\cdot q_{g',d-1}(0,j)}{k}\big)$.
Given $\E_1$, every edge $g' \in G \in \Gamma_T(g)$ either satisfies
$\beta(g') < jn^{-(v_H-2)/(e_H-1)}$ or else, $\beta(g')$ is uniformly distributed
in the interval $[jn^{-(v_H-2)/(e_H-1)}, kn^{-(v_H-2)/(e_H-1)}]$. Hence, 
since the number of choices for $g' \in G \in \Gamma_T(g)$ is at most $2\lambda e_H k^{e_H-1}$,
for $n$ sufficiently large, $\probcnd{\E_2}{\E_1} \ge (1-2\delta/k)^{2\lambda e_H k^{e_H-1}}$ 
which is at least $1-\epsilon_{d-1}$ by the choice of $\delta$. This proves the claim.
\item[(3)] The proof is similar to the proof of Claim~\ref{prop:facts}~(2).
\item[(4)] The proof is similar to the proof of Claim~\ref{prop:facts}~(1).
\end{enumerate}
\end{proof}

We continue with the proof of Proposition~\ref{prop:final}, which is by induction
on $d$. The base case, $d=0$ is trivially true, since
$q_{g,0}(j,j+\delta) = p_0(j,j+\delta) = 1$ for all $j \in J$. Fix $d \in [D]$,
assume that the proposition holds for $d-1$ and let $j \in J$.
In the argument that follows, we use implicitly the fact that for $y>1$, $\exp(-1/(y-1)) < 
1-1/y < \exp(-1/y)$.

By Claim~\ref{prop:facts}~(1), the induction hypothesis and Proposition~\ref{prop::pa1} and
by the occurrence of $E_1$ we have
\begin{eqnarray*}
q_{g,d}(j,j+\delta) &\le&
   \prod_{G \in \Gamma_T(g)}\Big(1 - \prod_{g' \in G}\frac{j \cdot q_{g',d-1}(0,j)}{k}\Big) \\
&\le&
   \prod_{G \in \Gamma_T(g)}\Big(1 - \prod_{g' \in G}\frac{(1-\epsilon_{d-1}) \cdot j \cdot p_{d-1}(0,j)}{k}\Big) \\
&\le&
   \Big(1 - \Big(\frac{(1-\epsilon_{d-1}) \cdot j \cdot p_{d-1}(0,j)}{k}\Big)^{e_H-1}
   \Big)^{\floor{\lambda k^{e_H-1}}- \floor{k^{e_H/2-1/3}}} = (*).
\end{eqnarray*}
Using Claim~\ref{prop:facts}~(3) we can further upper bound~$(*)$ as follows:
\begin{eqnarray*}
(*) &\le& \Big(1 - \Big(\frac{j \cdot p_{d-1}(0,j)}{k}\Big)^{e_H-1}
          \Big)^{(\floor{\lambda k^{e_H-1}}-\floor{k^{e_H/2-1/3}})(1-2e_H\epsilon_{d-1})} \\
&\le& \Big(1 - \Big(\frac{j \cdot p_{d-1}(0,j)}{k}\Big)^{e_H-1}
      \Big)^{\floor{\lambda k^{e_H-1}} - \floor{k^{e_H/2-1/3}} - 2e_H\epsilon_{d-1}\lambda k^{e_H-1}}  \\
&\le& (1+\epsilon_{d-1}) \cdot p_{d}(j,j+\delta) \cdot 
   \Big(1- \Big(\frac{j}{k}\Big)^{e_H-1}\Big)^{-2e_H\epsilon_{d-1}\lambda k^{e_H-1}} \\
&\le& (1+\epsilon_d/2) \cdot p_{d}(j,j+\delta).
\end{eqnarray*}
Now, by Claim~\ref{prop:facts}~(2), the induction hypothesis and Proposition~\ref{prop::pa1} and by the
occurrence of $E_1$ we have
\begin{eqnarray*}
q_{g,d}(j,j+\delta) &\ge&
(1-\epsilon_{d-1}) \cdot
        \prod_{G \in \Gamma_T(g)}\Big(1 - \prod_{g' \in G}\frac{j \cdot q_{g',d-1}(0,j)}{k}\Big) \\
&\ge&
(1-\epsilon_{d-1}) \cdot
        \prod_{G \in \Gamma_T(g)}\Big(1 - \prod_{g' \in G}\frac{(1+\epsilon_{d-1}) \cdot j \cdot p_{d-1}(0,j)}{k}\Big) \\
&\ge&
(1-\epsilon_{d-1}) \cdot
        \Big(1 - \Big(\frac{(1+\epsilon_{d-1}) \cdot j \cdot p_{d-1}(0,j)}{k}\Big)^{e_H-1}\Big)^{\lambda k^{e_H-1}+k^{e_H/2-1/3}} = (**).
\end{eqnarray*}
One can now use Claim~\ref{prop:facts}~(4) to further lower bound~$(**)$ as follows:
\begin{eqnarray*}
(**) &\ge&
(1-\epsilon_{d-1}) \cdot
\Big(1 - \Big(\frac{j \cdot p_{d-1}(0,j)}{k}\Big)^{e_H-1}
\Big)^{(\lambda k^{e_H-1}+k^{e_H/2-1/3})(1+2e_H\epsilon_{d-1})} \\
&\ge&
(1-\epsilon_{d-1}) \cdot
\Big(1 - \Big(\frac{j \cdot p_{d-1}(0,j)}{k}\Big)^{e_H-1}
\Big)^{\floor{\lambda k^{e_H-1}}-\floor{k^{e_H/2-1/3}}+8e_H\epsilon_{d-1}\lambda k^{e_H-1}} \\
&\ge&   
(1-\epsilon_{d-1}) \cdot p_d(j,j+\delta) \cdot
\Big(1 - \Big(\frac{j}{k}\Big)^{e_H-1}\Big)^{8e_H\epsilon_{d-1}\lambda k^{e_H-1}}  \\
&\ge& (1-\epsilon_d/2) \cdot p_d(j,j+\delta).
\end{eqnarray*} 
This completes the proof of Proposition~\ref{prop:final}.

\section*{Acknowledgment}
The author would like to kindly thank Ilan Newman and Oren Ben-Zwi for 
helpful discussions.




\section*{Errata}
\begin{itemize} 
\item At the end of Section~\ref{sec::2} we state that
Lemmas~\ref{lemma::toprove1}~and~\ref{lemma::toprove2} imply~(\ref{eq::main}).
For that implication to be valid, we also need that Lemma~\ref{lemma::toprove1}
would state that $\prob{E_1} = 1 - o(c n^{-(v_H-2)/(e_H-1)})$. This is exactly
what we prove in a subsequent section. The implication now holds essentially
since the above gives us that the probability of $E_1$ given $\event{\beta(f) <
cn^{-(v_H-2)/(e_H-1)}}$ is at least $1-o(1)$ and in addition, it follows from
Lemma~\ref{lemma::toprove1} and
Definitions~\ref{def::afdtree}~and~\ref{def::goodtree} that the probability of
$E_2(f)$ given $\event{\beta(f) < cn^{-(v_H-2)/(e_H-1)}}$ is also at least
$1-o(1)$.
\item Page $2$, line $-13$: ``than'' should be ``that''.  
\item
Page $13$, line $5$: ``the'' should be ``that''.  
\item Page $16$, line $13$:
``no no'' should be ``no''.  
\item Page $17$, line $-12$: ``There'' should be ``Given $g$, there''.
\item Page $18$, line $13$: ``edges'' should be ``vertices''.
\item Page $22$, line $3$: ``$T_{f,D-1}$'' should
be ``$RT_{f,D-1}$''.  
\end{itemize}

\end{document}